\documentclass[a4paper,12pt]{article}
\usepackage{amsmath, amsthm, amsfonts, amscd}
\usepackage{hyperref}
\newtheorem*{thrm}{Main Theorem}
\newtheorem{thm}{Theorem}
\newtheorem{cor}[thm]{Corollary}
\newtheorem{lem}[thm]{Lemma}
\newtheorem{defn}[thm]{Definition}
\newtheorem{prop}[thm]{Proposition}

\newcommand{\nc}{\newcommand}
\nc{\bea}{\begin{eqnarray*}} \nc{\eea}{\end{eqnarray*}} \nc{\map}{\to} \nc{\mch}{{\mathcal H}}
\nc{\mck}{{\mathcal K}} \nc{\mcc}{{\mathcal C}} \nc{\mcf}{{\mathcal F}} \nc{\mcm}{{\mathcal M}} \nc{\mco}{{\mathcal
O}} \nc{\mcd}{{\mathcal D}} \nc{\mcw}{{\mathcal W}} \nc{\mcl}{{\mathcal L}} \nc{\mbg}{{\mathfrak g}} \nc{\mbq}{{\mathfrak
q}} \nc{\mbh}{{\mathfrak h}} \nc{\cpx}{{\mathbb C}} \nc{\rhs}{rational homogeneous space } \nc\proj{\mathbb{P}}
\nc\pone{\mathbb{P}^1} \nc{\mrcs}{minimal rational curves } \nc{\mrc}{minimal rational curve } \nc{\vmrt}{variety of
minimal rational tangents } \nc{\pf}{\textit{Proof.} } \nc{\rank}{\text{rank}} \nc\fin{\hfill$\Box$}

\begin{document}

\begin{center}
{\bf\Large{Holomorphic maps from rational homogeneous\\spaces onto projective manifolds}}\\
\vspace{0.5cm} Chihin Lau \vspace{0.5cm}
\end{center}

\begin{abstract}
In \cite{rv}, Remmert and Van de Ven conjectured that if $X$ is the image of a surjective holomorphic map from
$\proj^n$, then $X$ is biholomorphic to $\proj^n$. This conjecture was proved by Lazarsfeld \cite{la} using Mori's
proof of Hartshorne's conjecture \cite{mr}. Then Lazarsfeld raised a more general problem, which was completely
answered in the positive by Hwang and Mok.
\begin{thm}\cite{hm2}\cite{hmb}\label{tl}
Let $S=G/P$ be a rational homogeneous manifold of Picard number 1. For any surjective holomorphic map $f:S\to X$ to a
projective manifold $X$, either $X$ is a projective space, or $f$ is a biholomorphism.
\end{thm}
The aim of this article is to give a generalization of Theorem \ref{tl}. We will show that modulo canonical
projections, Theorem \ref{tl} is true when $G$ is simple without the assumption on Picard number. We need to find a
dominating and generically unsplit family of rational curves which are of positive degree with respect to a given nef
line bundle on $X$. Such family may not exist in general but we will prove its existence under certain assumption
which is applicable in our situation.
\end{abstract}

\section*{1. Main statement}

(1.1) When $S$ is a hyperquadric, Theorem \ref{tl} is proved by Paranjape and Srinivas \cite{ps} in characteristic 0
and by Cho and Sato \cite{cs} in arbitrary characteristic. Tsai \cite{ts} proved the case when $S$ is a Hermitian
symmetric space. The complete statement is proved by Hwang and Mok \cite{hm2}\cite{hmb}. The aim of this article is to
generalize Theorem \ref{tl} to the case when $G$ is simple and $S$ is of higher Picard number. This means that $S=G/Q$
where $Q$ is a parabolic but not maximal parabolic subgroup of $G$. In this case there are of course non-finite
holomorphic maps, namely the canonical projections from $S$ onto rational homogeneous spaces of smaller Picard number
which correspond to parabolic subgroups $Q'$ strictly containing $Q$. If $f:S\to X$ is a surjective holomorphic map
onto a projective manifold which is not finite, we will show that $f$ factors through a canonical projection to a
finite map. When $f$ is finite, we will prove that either $X$ is a projective space or $f$ is a biholomorphism. Here
is the main result of this article.
\begin{thrm}
Let $G$ be a connected, simply-connected simple complex Lie group, $Q\subset G$ be a parabolic subgroup, $S=G/Q$ be
the corresponding rational homogeneous space, $\dim (S)=n$, and $f:S\to X$ be a holomorphic surjective map from $S$
to a projective manifold $X$. Then one of the following holds.
\begin{enumerate}
\item $f$ is a biholomorphism.
\item $f:S\to X$ is a finite map and $X$ is the projective space $\proj^n$.
\item There exists a parabolic subgroup $Q'$ of $G$ containing $Q$ as a proper subgroup such that $f$ factors through a finite map $g:G/Q'\to X$.
\end{enumerate}
\end{thrm}
In case (3), $f=g\circ\rho$, $\rho:G/Q\to G/Q'$ is an equivariant holomorphic fibration. Alternatives (1) and (2)
apply to $g:G/Q'\to X$, we have either $g:G/Q'\to X$ is a biholomorphism or $X$ is the projective space $\proj^m$,
$m=\dim(G/Q'$).\\

(1.2) Hwang and Mok gave two different proofs of Theorem \ref{tl}. In their first proof \cite{hm2}, they reduced the
theorem to the following extension problem.
\begin{prop}\cite{hm2}\label{pq}
Let $S$ be a rational homogeneous manifold of Picard number 1 different from ${\mathbb P}^n$ and $f:S\to X$ be a
finite holomorphic map to a projective manifold $X$ different from ${\mathbb P}^n$. Let $s,t\in S$ be an arbitrary
pair of distinct points such that $f(s)=f(t)$ and $f$ is unramified at $s$ and $t$. Write $\varphi$ for the unique
germ of holomorphic map at $s$, with target space $S$, such that $\varphi(s)=t$ and $f\circ\varphi=f$. Then $\varphi$
extends to a biholomorphic automorphism of $S$.
\end{prop}
Once Proposition \ref{pq} is proved, they concluded that $f:S\to X$ must be a quotient map by a finite group action,
$S$ is a hyperquadric and $X$ is the projective space. To prove Proposition \ref{pq}, they considered the so called
variety of minimal rational tangents $\mcc_x\subset\proj T_x(X)$ at a general point $x\in X$ outside the ramification
divisor $R$. If $X$ is not the projective space, they proved that $f^{-1}(\mcc_x)\subset\proj T_s(S)$ is a proper
$P$-invariant subset which is preserved by $\varphi$. Here for any point $s\in S$, $\proj T_s(S)$ is identified with
$\proj T_o(S)$, $o=eQ$, by $G$-action. On $S$ there is a canonical tower $\{0\}\subset D^1\subset
D^2\subset\cdots\subset D^m=T(S)$ of equivariant distributions. If $\varphi$ does not preserve any linearly degenerate
proper subset of $\proj T_o(S)$, then $S$ is a Hermitian symmetric space and $\varphi$ preserves the highest weight
orbit $\mcw$. Otherwise $\varphi$ preserves a proper $P$-invariant distribution which must be of the form $D^k$ for
some $k<m$. If $S$ is a contact space, then the Lagrangian property of the highest weight orbit $\mcw$ will imply that
$\varphi$ preserves $\mcw$. When $S$ is not a Hermitian symmetric space nor a contact space, they used the Lie algebra
structure to prove that $\varphi$ preserves $D^1$ if $\varphi$ preserves $D^k$ for some $k<m$. Then Proposition
\ref{pq} follows from a theorem of Yamaguchi (Proposition \ref{ac}) on differential system.

The second proof of Theorem \ref{tl} used no Lie theory. If $X$ is not the projective space, one can choose a
non-linear $\mcc_x$ at general $x\in X$. Essential to \cite{hmb} is the proof that any holomorphic vector field on $S$
descends to a holomorphic vector field on $X$ by $df$ if $\mcc_x$ is non-linear. The latter relies on a general
Cartan-Fubini extension principle for Fano manifolds of Picard number 1.

Since the Cartan-Fubini extension principle of \cite{hmb} does not in general apply in the case of higher Picard
number, we adopt the approach of \cite{hm2}. When $S$ is of Picard number larger than 1, the equivariant distributions
on $S$ are multi-graded. An equivariant distribution on $S$ is integrable if it is the relative tangent bundle of the
fibers of a canonical projection. The Main Theorem is proved in the following way. The case when $f$ is not finite is
easily reduced to the case when $f$ is finite. Suppose that the Main Theorem is false, let $l$ be the minimum positive
integer such that there exist a rational homogeneous space of Picard number $l$ and a finite ramified holomorphic map
$f:S\to X$ from $S$ to a projective manifold $X$ different from $\proj^n$. Following the approach of Hwang and Mok's
proof of Theorem \ref{tl} in \cite{hm2}, we consider and analyze the pull-back $df^{-1}(\mcc_x)$ of varieties of
minimal rational tangents (c.f. (2.1)) on $X$. The induced intertwining map $\varphi$, as defined in Proposition
\ref{pq}, preserves a proper $Q_s$-invariant subset of $\proj T_s(S)$. Then $\varphi$ also preserves the linear span
of $df^{-1}(\mcc_x)$. In the higher Picard number case, the main difficulty comes from the existence of proper
equivariant integrable distributions on $S$. We follow the method in \cite{hm2} to show that either $\varphi$
preserves $D^1$ or $\varphi$ preserves a proper equivariant integrable distribution $D$ (Proposition \ref{aj}). In the
former case, we either apply Yamaguchi's Theorem (Proposition \ref{ac}) as in \cite{hm2} if it is applicable or reduce
the problem to the latter case when Yamaguchi's Theorem cannot be applied. In the latter case, $f$ induces holomorphic
maps $\psi:X\to X'$ and $f':S'\to X'$ such that the diagram
$$\begin{CD}
S@>f>>X\\
@V\pi VV @VV\psi V\\
S'@>f'>> X'
\end{CD}\eqno{(\dagger)}$$
is commutative, where $\pi$ is a canonical projection with $D=\ker(d\pi)$.

In order to proceed, we need to find a dominating and generically unsplit family of rational curves which are
transversal to the fibers of $\psi$ generically, i.e., of positive degree with respect to $\psi^*\mcl'$ where $\mcl'$
is an ample line bundle on $X'$. We will call it a \textit{$\psi^*\mcl'$-effective minimal rational component}. The
standard Mori's bend-and-break argument does not guarantee the existence of such family. The difficulty is that a free
rational curve may split in such a way that the free components lie on the fibers and the transversal components are
not free. The major input of the present proof is the existence of such a family under the condition that $f$
restricted on the general fiber of $\pi$ is unramified or that $f'$ is bijective. When either of these conditions
holds, we can find an equivariant distribution preserved by $\varphi$ which is transversal to $D$. The first condition
enable us to replace $(\dagger)$ by a similar diagram until $f$ restricted on the general fiber of $\pi$ is ramified.
Using the second condition and a base extension argument, we prove that the image under $f$ of the general fiber of
$\pi$ cannot be a projective space. Thus a contradiction to the minimality of $l$ is obtained.

When $S=S_1\times S_2$, where $S_1$ and $S_2$ are rational homogeneous spaces associated with simple Lie groups, there
are holomorphic maps on $S$ where the analogue of Main Theorem does not hold. For example, we may take two holomorphic
maps $f_i:S_i\to X_i$, $i=1,2$, such that at least one of them is ramified and consider their Cartesian product
$f=(f_1,f_2)$. Then $f$ is not an isomorphism and its image is not the projective space. Thus the assumption that $G$
is simple in the statement of Main Theorem is necessary. One may ask whether there are holomorphic maps other than
this type violating analogue of Main Theorem. When studying this problem, one has to deal with the possibility that
$df^{-1}(\mcc_x)$ is equal to the union of two linear subspaces $\proj
T_s(S_1)\cup\proj T_s(S_2)$. Our method cannot be applied and new ingredient will be needed in this case.\\

(1.3) When $f$ is not finite, the following proposition from \cite{bl} reduces the problem to the case when $f$ is finite.
\begin{prop}(Proposition I.1 of \cite{bl})
Let $f:S\to X$ be a surjective holomorphic map from $S$ onto a projective manifold $X$. Then there exists a subgroup
$Q'$ of $G$ containing $Q$ such that $f$ induces a finite map $g:G/Q'\to X$.
\end{prop}

\section*{2. Variety of minimal rational tangents}
We present some basic properties of rational curves on projective manifolds and refer the reader to \cite{ko} and
\cite{hmv} for a detailed account. We will also define a notion of \textit{$\mcl_0$-effective minimal rational
tangents} for a nef line bundle $\mcl_0$.\\

(2.1) Let $X$ be an $n$-dimensional projective manifold. Let $\text{Hom}_{bir}(\proj^1,X)$ be the scheme parametrizing
morphisms from $\proj^1$ to $X$ which are birational onto its image and $\text{Hom}^n_{bir}(\proj^1,X)$ be its
normalization. Let $\text{RatCurves}(X)\subset\text{Chow}(X)$ be the quasi-projective subvariety whose points
correspond to irreducible and generically reduced rational curves on $X$ and $\text{RatCurves}^n(X)$ be its
normalization. The automorphism group $Aut(\proj^1)$ acts on $\text{Hom}^n_{bir}(\proj^1,X)$ naturally
and we have the following commutative diagram
$$\begin{CD}
\text{Hom}^n_{bir}(\proj^1,X)\times\proj^1@>quotient\ by>\text{Aut}(\proj^1)>\text{Univ}^{rc}(X)@>\iota>>X\\
@VprojectionVV@V\pi VV\\
\text{Hom}^n_{bir}(\proj^1,X)@>quotient\ by>\text{Aut}(\proj^1)>\text{RatCurves}^n(X)
\end{CD}.$$
Let $[l]\in\text{RatCurves}^n(X)$ and $\phi:\proj^1\to l$ be a normalization of $l$.
$\phi^*T(X)\cong\mco(a_1)\oplus\cdots\oplus\mco(a_n)$ by the Grothendieck splitting theorem. When $\phi^*T(X)$ is
semipositive, i.e. $a_i\geq0$, we say that $[l]$ is a \textit{free} rational curve. $X$ is said to be
\textit{uniruled} if it possesses a free rational curve. A point $x\in X$ in a uniruled projective manifold is called
a \textit{very general point} if every rational curve passing through $x$ is free. We say that $[l]$ is
\textit{standard} if $\phi^*T(X)$ is isomorphic to $\mco(2)\oplus\mco(1)^p\oplus\mco^q$.

Let $\mch\subset\text{RatCurves}^n(X)$ be an irreducible component and denote $\mch_x=\mch\cap\pi(\iota^{-1}(x))$. We
say that $\mch$ is \textit{dominating} if $\iota|_{\pi^{-1}(\mch)}$ dominates $X$. An irreducible component
$\mch\subset\text{RatCurves}^n(X)$ is dominating if and only if $[l]$ is free for a general element $[l]\in\mch$. A
dominating component is said to be \textit{unsplit} if it is proper. It is called \textit{generically unsplit} if
$\mch_x$ is proper for a general point $x\in X$. We will call $\mch$ a \textit{minimal rational component} if it is
dominating and generically unsplit. By Mori's bend-and-break argument a general member of a minimal rational component
$\mch$ is a standard rational curve. However, a dominating component of $\text{RatCurves}^n(X)$ whose general member
is standard may fail to be a minimal rational component.

Suppose $\mch$ is a minimal rational component. At a general point $x\in X$, every member of $\mch_x$ is a free
rational curve. Hence, the normalization $\widetilde{\mch_x}$ of $\mch_x$ is smooth. Moreover, a general member of
$\mch_x$ is standard. Consider the \textit{tangent map} $\tau_x:\tilde\mch_x--\to\proj T_x(X)$ which sends a curve
that is smooth at $x$ to its tangent direction at $x$. We define the \textit{variety of minimal rational tangents}
$\mcc_x$ associated with $\mch$ at $x$ to be the proper transform of $\mch_x$ under the tangent map. $\tau_x$ is
regular and of maximal rank at a point of $\mch_x$ corresponding to a standard rational curve, so that $\mcc_x$ is of
the same dimension as $\mch_x$. Kebekus showed that at a general point $x\in X$ any member of $\mch_x$ is immersed at
$x$ (Theorem 3.3 \cite{ke1}), and the tangent map $\tau_x$ is a finite morphism (Theorem 3.4 \cite{ke1}). Combining a
result of Hwang and Mok \cite{hmb} that at a general point $x\in X$, $\tau_x$ is also birational onto its image, we
know that the normalization of $\mcc_x$ is smooth, but not necessarily irreducible, for a general point $x$.

The most useful example of variety of minimal rational tangents is the one associated with $\mch$ of minimal degree.
Let $\mcl$ be an ample line bundle on $X$. The degree of any member of an irreducible component $\mch$ with respect to
$\mcl$ is the same. We call this the degree of $\mch$ with respect to $\mcl$. If $\mch$ is of minimal degree with
respect to $\mcl$ among all dominating components, then $\mch$ is a minimal rational component. When $X$ is of Picard
number larger than 1, we require something more.
\begin{defn}\label{da}
Let $\mcl_0$ be a nef line bundle on $X$. A $\mcl_0$-effective minimal rational component
$\mch\subset\text{RatCurves}^n(X)$ is an irreducible component of rational curves such that
\begin{enumerate}
\item $\mch$ is a minimal rational component.
\item $\mch$ is of positive degree with respect to $\mcl_0$.
\end{enumerate}
\end{defn}
An obvious way to try to get $\mcl_0$-effective minimal rational component is to consider component of minimal degree
with respect to a fixed ample line bundle $\mcl$ among all dominating components of positive degree with respect to
$\mcl_0$. But a component obtained in this way may fail to be generically unsplit. For example, consider the
Hirzebruch surface of genus one $F_1=\proj(\mco_{\proj^1}(1)\oplus\mco_{\proj^1})$ and the pull-back $\pi^*H$ of the
hyperplane bundle on $\proj^1$, where $\pi:F_1\to\proj^1$ is the natural projection. Then a component obtained by the
above process is not a minimal rational component. In fact there is no $\pi^*H$-effective minimal rational component
on $F_1$. Therefore a $\mcl_0$-effective minimal rational component may fail to exist in general. Fortunately in our
situation, we are able to prove its existence under some additional assumptions.

The variety of minimal rational tangents is important in studying Lazarsfeld's problem because of the following
proposition from \cite{hm2}.
\begin{prop}\label{pr}
Let $f:S=G/Q\to X$ be a finite morphism from a rational homogeneous space $S$ to a projective manifold $X$ different
from $\proj^n$. Let $\mch$ be a minimal rational component on $X$ and $\mcc_x$ be the variety of minimal rational
tangents associated with $\mch$ at $x=f(s)$. Then each irreducible component of $df^{-1}(\mcc_x)\subset\proj T_s(S)$
is $Q$-invariant for a general point $s\in S$.
\end{prop}

(2.2) Now we assume that there is a commutative diagram
$$\begin{CD}
S@>f>>X\\
@V\pi VV @VV\psi V\\
S'@>f'>> X'
\end{CD}\eqno{(\dagger)}$$
as in (1.2). Fix an ample line bundles $\mcl'$ on $X'$. We have the following criterion for the existence of a
$\psi^*\mcl'$-effective minimal rational component.
\begin{prop}\label{apa}
Suppose there is commutative diagram $(\dagger)$. Let $\mcl$ and $\mcl'$ be ample line bundles on $X$ and $X'$
respectively. Let $\mch\subset\text{RatCurves}^n(X)$ be an irreducible dominating component of positive degree with
respect to $\psi^*\mcl'$ which is of minimal degree with respect to $\mcl$ among all dominating components of positive
degree with respect to $\psi^*\mcl'$. If $f$ restricted on generic fiber of $\pi$ is unramified, then $\mch$ is a
$\psi^*\mcl'$-effective minimal rational component.
\end{prop}
\begin{proof}
It suffices to prove that $\mch$ is generically unsplit. Let $R$ be the ramification divisor of $f$ and $x\in X$
be a point outside $f(R)$. Let $\{C_\lambda\}$ be an algebraic family of rational curves parametrized by some
algebraic curve $B$ such that $C_\lambda\in\mch_x$ generically. We need to show that all $C_\lambda$ are irreducible
curves. Suppose for some $\lambda$, $C_\lambda=C_1+C_2+\cdots+C_k$ is reducible. Since $C_\lambda$ is of positive
degree with respect to $\psi^*\mcl'$, one of the components, say $C_1$, is of positive degree with respect to
$\psi^*\mcl'$. Now $R=\pi^*R'$, where $R'\subset S'$ is the ramification divisor of $f'$, and $f(R)$ is disjoint from
the fiber of $\psi$ which contains $x$. We may further assume that $C_1$ contains a point on the fiber of $\psi$
containing $x$. Thus $C_1$ is not contained in $f(R)$. We claim that $C_1$ is free. Otherwise, $T(X)|_{C_1}$ is not
semi-positive and $H^0(C_1,\mco(-1)\otimes T^*(X)|_{C_1})\neq0$. Since $df$ is bijective at a generic point of
$f^{-1}(C_1)$, $H^0(f^{-1}(C_1),f^*\mco(-1)\otimes T^*(S)|_{f^{-1}(C_1})\neq0$. But $T(S)$ restricted to any curve on
$S$ is nef by the homogeneity of $S$. This is impossible and hence $C_1$ is a free rational curve of positive degree
with respect to $\psi^*\mcl'$ with ${\mathcal L}\cdot C_\lambda={\mathcal L}\cdot(C_1+C_2+\cdots+C_k)>{\mathcal
L}\cdot C_1$. This contradicts the minimality of degree of $\mch$ with respect to $\mcl$ among all dominating
component of positive degree with respect to $\psi^*\mcl'$. Therefore all $C_\lambda$ are irreducible.
\end{proof}

\section*{3. Rational homogeneous spaces and equivariant distributions on them}

The basic references to this section are \cite{ya} and \cite{hm2}.\\

(3.1) Let $G$ be a connected and simply-connected simple complex Lie group. A parabolic subgroup $Q$ of $G$ is a
closed Lie subgroup such that the space of left coset $G/Q$ is projective-algebraic. The quotient $S=G/Q$ is called a
\textit{rational homogeneous space}. We will only consider the case where $G$ is simple. Let $\mbg$ be the Lie algebra
of $G$ and $\mbq$ be the parabolic subalgebra corresponding to $Q$. Fix a Levi decomposition ${\mathfrak q}={\mathfrak
u}+{\mathfrak l}$, where ${\mathfrak u}$ is nilpotent and ${\mathfrak l}$ is reductive, and a Cartan subalgebra
${\mathfrak h}\subset{\mathfrak l}$, which is also a Cartan subalgebra of $\mbg$. We have the root system
$\Phi\subset{\mathfrak h}^*$ of $\mbg$ with respect to ${\mathfrak h}$. We can choose a set of positive roots uniquely
by requiring that ${\mathfrak u}$ is contained in the span of negative root spaces. (Our sign conventions follow
\cite{hm2} so that positive roots correspond to positive line bundles and are different from those of other
references, e.g. from \cite{ya}.) Fix a system of simple roots $\Delta=\{\alpha_1,\cdots,\alpha_r\}$. Let
$\Delta_0=\{\alpha\in\Delta:\alpha({\mathfrak z})=0\}$ where ${\mathfrak z}$ is the center of ${\mathfrak l}$ and
$\Delta_1=\Delta\setminus\Delta_0=\{\alpha_{r_1},\cdots,\alpha_{r_l}\}$. Then ${\mathfrak z}_i=\{z\in{\mathfrak
z}:\alpha_{r_j}(z)=0$, for any $j\neq i\}$, $1\leq i\leq l$, are one-dimensional subalgebras of ${\mathfrak z}$ and
$\{{\mathfrak z}_1,\cdots,{\mathfrak z}_l\}$ generate ${\mathfrak z}$. We say that $\mbq$ is the parabolic subalgebra
associated to $\Delta_1$, and
$S$ is of type $(\mbg,\Delta_1)$.\\
Define \bea
\Phi_{k_1,\cdots,k_l}&=&\{\alpha\in\Phi:\alpha=\sum\limits^r_{j=1}a_j\alpha_j,\ a_{r_i}=k_i\}\\
\Phi_k&=&\bigcup_{k_1+\cdots+k_l=k}\Phi_{k_1,\cdots,k_l}\\
\Phi^+&=&\bigcup_{k>0}\Phi_k \eea and \bea
\mbg_0&=&{\mathfrak h}\oplus\bigoplus_{\alpha\in\Phi_0}\mbg_\alpha\\
\mbg_k&=&\bigoplus_{\alpha\in\Phi_k}\mbg_\alpha,\ k\neq0\\
\mbg_{k_1,\cdots,k_l}&=&\bigoplus_{\alpha\in\Phi_{k_1,\cdots,k_l}}\mbg_\alpha. \eea Then $\mbg_k$
is an eigenspace for the adjoint representation of ${\mathfrak z}$. In fact there exists elements $\theta_i\in{\mathfrak
z}_i$ such that $[\theta_i,v]=k_iv$ for $v\in\mbg_{k_1,\cdots,k_l}$. The eigenspace decomposition
$$\mbg=\bigoplus^m_{k=-m}\mbg_k=\bigoplus_{(k_1,\cdots,k_l)\in{\mathbb Z}^l}\mbg_{k_1,\cdots,k_l}$$
gives a multi-graded Lie algebra structure on $\mbg$. Here the depth $m_i$ of the $i$-th node $\alpha_{r_i}$ is the
largest integer such that $\Phi_{\cdots,m_i,\cdots}\neq\emptyset$, $i$-th entry is $m_i$, and the total depth $m$ is
the largest integer such that $\Phi_m\neq\emptyset$. If $m=1$, then $l=1$ and $S$ is a Hermitian symmetric space. If
$m=2$, $\dim\mbg_2=1$ ,then the bracket $[,]:\mbg_1\times\mbg_1\to\mbg_2$ is non-degenerate and $S$ is a contact
space. We have \bea
\mbq&=&\mbg_0\oplus\mbg_{-1}\oplus\cdots\oplus\mbg_{-m}\\
{\mathfrak l}&=&\mbg_0\\
{\mathfrak u}&=&\mbg_{-1}\oplus\cdots\oplus\mbg_{-m}. \eea The tangent space $T_o(S)$, $o=eQ$, can be identified with
$\mbg/\mbq$ canonically, which can be identified with $\bigoplus\limits_{\alpha\in\Phi^+}\mbg_\alpha$ by a choice of a
Levi factor $L\subset Q$. $\bigoplus\limits_{k'_i\leq k_i}\mbg_{k'_1,\cdots,k'_l}$ define a family of $Q$-invariant
vector subspaces of $T_o(S)$. Under the action of $G$ they correspond to equivariant holomorphic distributions
$D^{k_1,\cdots,k_l}$ on $S$ such that $D^{k_1,\cdots,k_l}_o=\bigoplus\limits_{k'_i\leq k_i}\mbg_{k'_1,\cdots,k'_l}$.
Denote $D^k=\sum\limits_{k_1+\cdots+k_l=k}D^{k_1,\cdots,k_l}$. These distributions are partially ordered by inclusion
which corresponds to the natural order on ${\mathbb Z}^l$.

Let $W^{k_1,\cdots,k_l}\subset{\mathbb P}\mbg_{k_1,\cdots,k_l}$ be the orbit of highest weight vectors of $L$-action
on $\mbg_{k_1,\cdots,k_l}$ and ${\mathcal W}^{k_1,\cdots,k_l}\subset{\mathbb P}D^{k_1,\cdots,k_l}\subset{\mathbb
P}T(S)$ is a fiber subbundle defined by the union of the translates of $W^{k_1,\cdots,k_l}$. Note that
$W^{k_1,\cdots,k_l}$ depends on the choice of Levi subgroup $L$ while ${\mathcal W}^{k_1,\cdots,k_l}$ does not.

For $s\in S$ we denote by $Q_s\subset G$ the parabolic subgroup fixing $s$. Denote by $U_s\subset Q_s$ the unipotent
radical, $L_s=Q_s/U_s$, and regard $D^{0,\cdots,0,1,0,\cdots,0}_s$ as an $L_s$-representation space. The set of all
highest weight vectors $\mcw^{0,\cdots,0,1,0,\cdots,0}_s\subset\proj D^{0,\cdots,0,1,0,\cdots,0}_s$ is a rational
homogeneous space and $L_s$ acts transitively on it. The collection of $\mcw^{0,\cdots,0,1,0,\cdots,0}_s$ as $s$ range
over $S$ defines a homogeneous holomorphic fiber bundle $\mcw^{0,\cdots,0,1,0,\cdots,0}\to S$.

There are canonical structures of holomorphic fiber bundles defined on $S$. Let $S_i=G/P_i$, $1\leq i\leq l$, where
$P_i$ is the maximal parabolic subgroup corresponding to ${\mathfrak
p}_i=\bigoplus\limits_{k_i\leq0}\mbg_{k_1,\cdots,k_l}$. $S_i$ is a rational homogeneous space of type $(\mbg,\{\alpha_{r_i}\})$. We have
canonical projection $\pi_i:S\to S_i$ induced by the inclusion $Q\subset P_i$. More generally for every
$\Delta'\subset\Delta_1$, let $\mbq_{\Delta'}\supset\mbq$ be the parabolic subalgebra corresponding to $\Delta'$
($\mbq_{\{\alpha_{r_i}\}}={\mathfrak p}_i$). We have canonical projections $\pi_{\Delta'}:S\to S_{\Delta'}$ induced
by the inclusion $Q\subset Q_{\Delta'}$, where $S_{\Delta'}=G/Q_{\Delta'}$ is a rational homogeneous space of type ($\mbg,\Delta'$) and
$Q_{\Delta'}$ ($Q_{\{\alpha_{r_i}\}}=P_i$) is the subgroup corresponding to $\mbq_{\Delta'}$.\\

(3.2) To describe all equivariant distributions, we need the following two elementary lemmas. The proofs are similar
and we will only prove the second one.
\begin{lem}\label{le}
Let $\alpha,\beta\in\Phi^+$ be two positive roots with $\alpha<\beta$, then there exists a sequence of simple roots
$\alpha_{i_1},\cdots,\alpha_{i_s}$ such that
$$\beta=\alpha+\alpha_{i_1}+\cdots+\alpha_{i_s}$$
and the partial sum
$$\alpha+\alpha_{i_1}+\cdots+\alpha_{i_t}$$
are all roots for any positive integer $1\leq t\leq s$.
\end{lem}
\begin{lem}\label{laa}
Let $\alpha,\beta\in\Phi_{k_1,\cdots,k_l}$, $(k_1,\cdots,k_l)>0$, be two positive roots, then there exists a sequence
of simple roots $\alpha_{i_1},\cdots,\alpha_{i_s}\in\Delta\setminus\Delta_1$ such that
$$\beta=\alpha+\epsilon_1\alpha_{i_1}+\cdots+\epsilon_s\alpha_{i_s},\ \epsilon_i\in\{1,-1\}$$
and the partial sum
$$\alpha+\epsilon_1\alpha_{i_1}+\cdots+\epsilon_t\alpha_{i_t}$$
are all roots in $\Phi_{k_1,\cdots,k_l}$ for any positive integer $1\leq t\leq s$. Let $L\subset Q$ be a Levi factor
of $Q$, then each $\mbg_{k_1,\cdots,k_l}$ is an irreducible $L$-module.
\end{lem}
\begin{proof}
Let $\{\alpha_1,\cdots,\alpha_k\}=\Delta$ be the set of all simple roots and $\beta-\alpha=\sum\limits_{i=1}^k
a_i\alpha_i$. Let $|\beta-\alpha|=\sum\limits_{i=1}^k |a_i|$. We will prove the lemma by induction on
$|\beta-\alpha|$. Suppose $|\beta-\alpha|=s+1$. Observe that
\bea
(\beta-\alpha,\beta-\alpha)&>&0\\
(\beta,\beta-\alpha)&>&(\alpha,\beta-\alpha). \eea So there exists $1\leq i\leq k$ such that $(\beta,a_i\alpha_i)>0$
or $(\alpha,a_i\alpha_i)<0$. Then we may take $\gamma=\beta-\epsilon\alpha_i$ if $(\beta,a_i\alpha_i)>0$ or
$\gamma=\alpha+\epsilon\alpha_i$ if $(\alpha,a_i\alpha_i)<0$, where $\epsilon=sign(a_i)$, so that $\gamma$ is a root.
It is easy to see that $|\gamma-\alpha|,|\gamma-\beta|\leq s$. By induction hypothesis, there exists sequence of
simple roots joining $\gamma$ with $\alpha$ and $\beta$ and the sequence of simple roots we needed exists. Hence there
is a highest weight $\eta$ in $\Phi_{k_1,\cdots,k_l}$ and every $\alpha\in\Phi_{k_1,\cdots,k_l}$ ascends to $\eta$.
Therefore each $\mbg_{k_1,\cdots,k_l}$ is an irreducible
$L$-module.
\end{proof}

We can now describe all equivariant distributions.
\begin{prop}\label{pk}
If $D$ is an equivariant distribution on $S$, then
$$D=\sum\limits_{(k_1,\cdots,k_l)\in\Lambda}D^{k_1,\cdots,k_l},$$
where $\Lambda\subset{\mathbb Z}^l$ is a finite subset with $\lambda>(0,\cdots,0)$ for all $\lambda\in\Lambda$.
\end{prop}
\begin{proof}
Any $Q$-invariant subset $V\subset T_o(S)$ is a $Q$-invariant vector subspace which is also a $L$-module. Then
$V=\bigoplus\limits_{(k_1,\cdots,k_l)\in\Lambda}\mbg_{k_1,\cdots,k_l}/\mbq$ for some finite set
$\Lambda\subset{\mathbb Z}^l$. Each $\mbg_{k_1,\cdots,k_l}$ is an irreducible $L$-module. By Lemma \ref{le}, we have
$$(ad\:\mbq)^s\mbg_{k_1,\cdots,k_l}=\bigoplus\limits_{s_i\leq k_i}\mbg_{s_1,\cdots,s_l},$$
for some positive integer $s$. Thus if $V$ contains $\mbg_{k_1,\cdots,k_l}/\mbq$, then it must contain
$$\bigoplus\limits_{(0,\cdots,0)<(s_1,\cdots,s_l)\leq (k_1,\cdots,k_l)}\mbg_{s_1,\cdots,s_l}/\mbq.$$
Let $\cpx^*_i\subset Q$, $1\leq i\leq l$, be the group corresponding to ${\mathfrak z}_i\subset{\mathfrak z}$, where
${\mathfrak z}_i=\{z\in{\mathfrak z}:\alpha_{r_j}(z)=0$, for any $j\neq i\}$, $1\leq i\leq l$ and ${\mathfrak z}$ is the
center of the Lie algebra ${\mathfrak l}$ of $L$. $t\in\cpx^*_i$ acts on $v\in T_o(S)$,
$v=\sum\limits_{(s_1,\cdots,s_l)\in\bar{\Lambda}}v_{s_1,\cdots,s_l}$, $v_{s_1,\cdots,s_l}\in\mbg_{s_1,\cdots,s_l}$ by
$t\cdot v=\sum\limits_{(s_1,\cdots,s_l)}t^{s_i}v_{s_1,\cdots,s_l}$. It follows that the closure of $Q$-orbit of
$[v]\in\proj T_o(S)$ contains an element in $\proj\mbg_{k_1,\cdots,k_l}$ if $v_{k_1,\cdots,k_l}\neq0$ and
$(k_1,\cdots,k_l)\in\bar{\Lambda}$ is a maximal element. Moreover a closed $Q$-orbit contains
$\proj\mbg_{k_1,\cdots,k_l}$ must contain $\proj\mbg_{s_1,\cdots,s_l}$ if $(s_1,\cdots,s_l)\leq(k_1,\cdots,k_l)$ by
Lemma \ref{laa}. Therefore $V$ must be of the form
$$\sum\limits_{(k_1,\cdots,k_l)\in\Lambda}(\bigoplus\limits_{(0,\cdots,0)<(s_1,\cdots,s_l)
\leq
(k_1,\cdots,k_l)}\mbg_{s_1,\cdots,s_l}/\mbq)=\bigoplus\limits_{(k_1,\cdots,k_l)\in\bar{\Lambda}}\mbg_{k_1,\cdots,k_l}/\mbq,$$
where $\Lambda\subset{\mathbb Z}^l$ is a finite subset which is the set of maximal elements of
$\bar{\Lambda}=\{\lambda\in{\mathbb Z}^l:(0,\cdots,0)<\lambda\leq\xi,\text{ for some }\xi\in\Lambda\}$. Hence
$$D=GV=\sum\limits_{(k_1,\cdots,k_l)\in\Lambda}D^{k_1,\cdots,k_l}.$$
This completes the proof of the proposition.
\end{proof}
\begin{cor}\label{ca}
Let $m$ be the depth of $S$. An equivariant distribution $D$ is a proper distribution if and only if $D\subset
D^{m-1}$. A sum of proper equivariant distributions is again a proper distribution.
\end{cor}
\begin{proof}
Since $G$ is simple, there exists unique highest root $\eta\in\Phi_{m_1,\cdots,m_l}$. Let $D$ be a proper
equivariant distribution. By Proposition \ref{pk}, $D=\sum\limits_{\lambda\in\Lambda}D^\lambda$ for some finite subset
$\Lambda\subset{\mathbb Z}^l$. Then $|\lambda|\leq m-1$ since $\lambda<(m_1,\cdots,m_l)$. Therefore $D\subset
D^{m-1}$. The second statement follows from the first.
\end{proof}

Let $D\subset T(S)$ be an equivariant distribution. Then $D$ is integrable if it is the relative tangent bundle of the
fibers of a canonical projection $S\to G/Q'$ induced by some parabolic subgroup $Q'$ containing $Q$. Conversely, if
$D$ is an equivariant integrable distribution, let $\mbq'\subset\mbg$ be a linear subspace such that $\mbq\subset\mbq'$
and $\mbq'/\mbq=D_o$. Then $\mbq'$ is a parabolic subalgebra since $[D_o,D_o]\subset D_o$ and $\mbq'$ contains $\mbq$.
Thus $D$ is the relative tangent bundle of the fibers of the projection $\pi:S\to S'=G/Q'$ where $Q'$ is the parabolic
subgroup corresponding to $\mbq'$. An equivariant distribution is called a \textit{minimal integrable distribution} if
it is integrable and is of the form $D^{0,\cdots,0,k_i,0,\cdots,0}$ for some $k>0$. The \textit{complementary
distribution} $D^c$ of $D$ is the smallest equivariant integrable distribution containing simple root vectors which are
not contained in $D$. An equivariant integrable distribution is a proper distribution if and only if it does not
contain $D^1$ since $D^1$ generates the whole tangent space $T(S)$ as a differential system. In this paper, we will
only consider equivariant distribution and will sometimes drop the word equivariant.

\section*{4. Induced fibration on the target manifold and rational curves transversal to them}

From now on, $S=G/Q$ will denote a rational homogeneous space and $f:S\to X$ is a finite surjective holomorphic map
onto projective manifold $X$. Let $s,t\in S$ be an arbitrary pair of distinct points such that $f(s)=f(t)$ and $f$ is
unramified at $s$ and $t$. Write $\varphi$ for the unique germ of holomorphic map at $s$, with target space $S$, such
that $\varphi(s)=t$ and $f\circ\varphi=f$. We will call $\varphi$ the induced intertwining map.

The main difficulty of generalizing Theorem \ref{pq} to higher Picard number case is that the equivariant distribution
$D$ spanned by the pull back $df^{-1}\mcc_x$ of a variety of minimal rational tangents at a general point $x$ on $X$
may be contained in a proper equivariant integrable distribution. In \cite{hm2}, the authors considered the rank of
$F_\eta$ where $F_\eta$ is a map defined from the Frobenius bracket induced by $D$ to recover $D^1$ from $D$ and conclude that $\varphi$ preserves $D^1$ in the Picard
number 1 case. This method cannot be applied directly in the current situation since on one hand $D$ may itself be
integrable so that the rank is always zero. On the other hand, when $S$ is of Picard number larger than 1, it may
happen that $D$ does not contain $D^1$ and it is impossible to recover $D^1$ from $D$ by considering the rank of $F_\eta$
only. In this section, we prove that if $D$ is contained in an integrable distribution, then there exists a
commutative diagram $(\dagger)$ as in Proposition \ref{aq}. Then the existence of a $\psi^*\mcl'$-effective minimal
rational component enable us to find a $\varphi$-preserving equivariant distribution which is transversal to $D$ under
certain assumption.\\

(4.1) We prove an analogue of Proposition 6 in \cite{hm2}.
\begin{prop}\label{ph}
If $S$ is of Picard number $l>1$ and $X$ is not $\proj^n$, then the induced intertwining map
$\varphi$ preserves a proper equivariant distribution $D$ of $S$.
\end{prop}
\begin{proof}
Identify $\proj T_s(S)$ and $\proj T_t(S)$ with $\proj T_o(S)$ by $G$-action. Let
$W^1=W^{1,0,\cdots,0}\cup\cdots\cup W^{0,\cdots,0,1}$ be the union of highest weight $L$-orbits in $\proj\mbg_1$ and
$A\in\proj GL(T_o(S))$ be the algebraic subgroup generated by $[d\varphi]$ and the image of $Q$ in $\proj GL(T_o(S))$
under the isotropy representation. By Proposition \ref{pr}, $\varphi$ preserves a proper $Q$-invariant subset
$df^{-1}(\mcc_x)$, where $x=f(s)=f(t)$. By the same proof as in Proposition 5 of \cite{hm2}, any linearly
non-degenerate invariant closed subvariety of $\proj T_o(S)$ must contain $W$. In particular, if $df^{-1}(\mcc_x)$
does not contain $W^1$, then the linear span of $df^{-1}(\mcc_x)$ is a proper $Q$-invariant subset of $\proj T_o(S)$
preserved under $\varphi$. Thus we may assume without loss of generality that $df^{-1}(\mcc_x)$ contains $W^1$. Then,
$A\cdot W^1\subset df^{-1}(\mcc_x)\subset\proj T_o(S)$ is a proper $Q$-invariant subset. Suppose $A\cdot W^1$ is
linearly non-degenerate. Then, one of the irreducible components of $A\cdot W^1$ is linearly non-degenerate in $\proj
T_o(S)$ since a direct sum of $Q$-invariant proper subspaces of $T_o(S)$ remains proper by Corollary \ref{ca}. It
follows that one of $E := A\cdot W^{0,\cdots,0,1,0,\cdots,0}$ is linearly non-degenerate. Now $B := \overline{E} - E$
is a constructible $A$-invariant subset. It's Zariski closure $\overline{B}\subset\overline{E}$ is a proper
$A$-invariant subvariety. If $\overline{B}$ is linearly non-degenerate, then $\overline{B}$ contains $W^1$ by the
observation made in the beginning, therefore it contains $E$ which is absurd. If $\overline{B}$ is a non-empty
linearly degenerate subset, then the linear span of $B \subset \proj T_o(S)$ gives a proper $Q$-invariant vector
subspace preserved by $\varphi$. Finally, if $B = \emptyset$ ,then $E \subset \proj T_o(S)$ is a proper connected
homogeneous submanifold. By the proof of Proposition 6 in \cite{hm2}, $S$ is a Hermitian symmetric space which must be
of Picard number 1 (see (3.1)). This
contradicts our assumption on $S$.
\end{proof}

If $D$ in the above proposition is contained in the kernel of a canonical projection $\pi:S\to S'$, we have an
induced fibration on $X$ by the following proposition.
\begin{prop}\label{aq}
If $\varphi$ preserves an integrable distribution $D=\ker(d\pi_{\Delta'})$ for some $\Delta' \subset\Delta_1$, then
there exists a surjective morphism $\psi:X\to X'$ over a normal variety $X'$ and a finite morphism $f':S'\to X'$
such that the diagram
$$\begin{CD}
S@>f>>X\\
@V\pi VV @VV\psi V\\
S'@>f'>> X'
\end{CD}\eqno(\dagger)$$
is commutative with $\pi=\pi_{\Delta'}$.
\end{prop}
\begin{proof}
Let $X'$ be an irreducible component of the Chow space of $X$ containing the reduced cycle $[f(\pi^{-1}(s'))]$ for
a general fiber of $\pi:S\to S'$. Let $\phi:{\mathcal F}\to X$, $\mu:{\mathcal F}\to X'$ be the universal family
morphism associated to $X'$. By taking reduced structures of the complex spaces involved and then taking
normalization, we may assume that ${\mathcal F}$ and $X'$ are reduced and normal. Since $f$ is finite, there exists
only a finite number of subvarieties through each point of $X$, which are the images under $f$ of fibers of $\pi$. It
follows that $\phi$ is a finite morphism. On the other hand, since $\phi$ parametrizes the leaves of a foliation, it
must be a birational map. Thus $\phi$ is an isomorphism and we get $\psi=\mu\circ\phi^{-1}:X\to X'$.
\end{proof}

(4.2) Now we describe line bundles and rational curves associated to roots of the Lie algebra. The following is taken
from \cite{hmv}. For $\rho\in\Phi^+$, let $H_\rho\in\mbh_\rho=[\mbg_\rho,\mbg_{-\rho}]$ be such that $\rho(H_\rho)=2$.
We call $H_\rho$ the coroot of $\rho$. Basis vectors $E_\rho\in\mbg_\rho$, $E_{-\rho}\in\mbg_{-\rho}$ can be chosen
such that $[E_\rho,E_{-\rho}]=H_\rho$, $[H_\rho,E_\rho]=2E_\rho$, $[H_\rho,E_{-\rho}]=-2E_\rho$ so that the triple
$(H_\rho,E_\rho,E_{-\rho})$ defined an isomorphism of ${\mathfrak s}_\rho=\mbh_\rho\oplus\mbg_\rho\oplus\mbg_{-\rho}$
with ${\mathfrak sl}_2(\cpx)$ \cite{se}. Let now $C_\rho\subset S$ be the $\proj SL(2,\cpx)$ orbit of $o=eQ$ under the
Lie group $S_\rho\cong\proj SL(2,\cpx)$, $S_\rho\subset G$ with Lie algebra ${\mathfrak s}_\rho$. Let $\alpha_{r_i}$
be the $i$-th simple root in $\Delta_1$ and $\omega_{r_i}$ be the $r_i$-th fundamental weight with
$\omega_{r_i}(H_{\alpha_j})=\delta_{r_ij}$. Let $E_i$ be the underlying vector space of the representation of $G$,
with lowest weight $-\omega_{r_i}$ and $\upsilon\in E_i$ be a lowest weight vector. We have
$H\upsilon=-\omega_{r_i}(H)\upsilon$ for any $H\in\mbh$. The stabilizer of the action of $G$ on $[\upsilon]\in\proj
E_i$ is precisely the maximum parabolic subgroup $P_i$. Therefore this action defines an embedding
$\tau:S_i=G/P_i\to\proj E_i$. Then ${\mathcal L}_i=\pi_i^*\tau^*({\mathcal O}(1))$, where $\pi_i$ is the canonical
projection and ${\mathcal O}(1)$ is the hyperplane bundle on $\proj E_i$, is a line bundle on $S$. For the rational
curve $C_\rho$ with $\omega_{r_i}(H_\rho)=k$ we have $H_\rho\upsilon=-k\upsilon$. Since $H_\rho$ is a generator of the
weight lattice of ${\mathfrak s}_\rho$, the pull-back of ${\mathcal O}(1)$ on $\proj E_i$ to $C_\rho$, which is the
dual of the tautological line bundle, gives a holomorphic line bundle $\cong{\mathcal O}(k)$. In particular, for
$\rho=\alpha_{r_j}$ we have $\omega_{r_i}(H_{r_j})=\delta_{ij}$ so that $L_j=C_{\alpha_{r_j}}\subset S$ represents a
rational curve of degree 1 with respect to ${\mathcal L}_j$ and of degree 0 with respect to ${\mathcal
L}_i$ for $i\neq j$.\\

(4.3) In the rest of this section, we assume that there exists a commutative diagram $(\dagger)$ as in Proposition
\ref{aq}. In Proposition \ref{apa}, we have given a criterion for the existence of a $\psi^*\mcl'$-effective minimal
rational component. Now we are going to give one more criterion. The following lemma will be used.
\begin{lem}\label{lh}
Let $\mcl$ and $\mcl'$ be ample line bundles on $X$ and $X'$ respectively. If $f'$ is bijective, then for any rational
curve $C$ on $X$ which is of positive degree with respect to $\psi^*\mcl'$, there exists a free rational curve $C_f$
on $X$ which is also of positive degree with respect to $\psi^*\mcl'$ such that $\mcl\cdot C_f\leq\mcl\cdot C$.
\end{lem}
\begin{proof}
Identify $X'$ with $S'$ by $f'$. Let $Q'\supset Q$ be the parabolic subgroup where $S'=G/Q'$ and $\mbq'\supset\mbq$
be the Lie algebra associated with $Q'$. Denote $\Delta$ the set of simple roots, $\Delta_1=\Delta\setminus\mbq$ and
$\Delta'=\Delta\setminus\mbq'\subset\Delta_1$ so that $S'$ is of type ($\mbg,\Delta'$). For any
$\alpha_{r_i}\in\Delta_1$, $1\leq i\leq l$, let $\mcl_i$ and $C_i$ be the line bundle and rational curve on $S$
associated with $\alpha_{r_i}$. For any $\alpha_{r_i}\in\Delta'$, let $\mcl'_i$ and $C'_i$ be the line bundle and
rational curve on $S'$ associated with $\alpha_{r_i}$. It is known that $\{\mcl_i\}_{1\leq i\leq l}$ and
$\{C_i\}_{1\leq i\leq l}$ are generators in $Pic(S)$ and $H_2(S,{\mathbb Z})$ \cite{be}. For any
$\alpha_{r_i}\in\Delta'$, we have $\pi^*\mcl'_i=\mcl_i$. Since $\mcl_i\cdot C_j=\delta_{ij}$, $\{C_i\}_{1\leq i\leq l}$
is the dual basis of $\{\mcl_i\}_{1\leq i\leq l}$. Let $a_i=\mcl_i\cdot f^*(C)\geq 0$, we have
$f^*(C)=\sum\limits^l_{i=1}a_iC_i$. Let $d$ be the degree of $f$. Since $C$ is of positive degree with respect to
$\psi^*\mcl'$, there exists $1\leq j\leq l$ such that $\alpha_{r_j}\in\Delta'\subset\Delta_1$ such that
$\mcl'_j\cdot\psi_*(C)>0$ and we have
$$a_j=\pi^*\mcl'_j\cdot f^*(C)=\psi^*\mcl'_j\cdot f_*(f^*(C))=d(\psi^*\mcl'_j\cdot C)=d(\mcl'_j\cdot\psi_*(C))\geq d.$$
This implies that
$$d(\mcl\cdot f_*(C_j))=d(f^*\mcl\cdot C_j)\leq f^*\mcl\cdot a_jC_j\leq f^*\mcl\cdot f^*(C)=\mcl\cdot f_*(f^*(C))=d(\mcl\cdot C).$$
To get $C_f$, apply $G$ action to $C_j$ if necessary so that $f_*(C_j)$ is free.
\end{proof}
\begin{prop}\label{ap}
Suppose there is commutative diagram $(\dagger)$ as in Proposition \ref{aq}. Let $\mcl$ and $\mcl'$ be ample line
bundles on $X$ and $X'$ respectively. Let $\mch\subset\text{RatCurves}^n(X)$ be an irreducible dominating component of
positive degree with respect to $\psi^*\mcl'$ which is of minimal degree with respect to $\mcl$ among all dominating
components of positive degree with respect to $\psi^*\mcl'$. If $f'$ is bijective, then $\mch$ is a
$\psi^*\mcl'$-effective minimal rational component.
\end{prop}
\begin{proof}
As in the proof of Proposition \ref{apa}, we show that $\mch$ is generically unsplit. Let $\{C_\lambda\}$ be an
one parameter family of rational curves such that $C_\lambda\in\mch$ generically. Assume that
$C_\lambda=C_1+\cdots+C_k$ is reducible for some $\lambda$ such that $C_1$ is a component of positive degree with
respect to $\psi^*\mcl'$. By Lemma \ref{lh}, there exists a free rational curve $C_f$ of positive degree with respect
to $\psi^*\mcl'$ such that $\mcl\cdot C_f\leq\mcl\cdot C_1<\mcl\cdot C$. This contradicts the minimality of
$\mch$.
\end{proof}

(4.4) If $f$ restricted on fibers of $\pi$ is unramified or $f'$ is bijective, then there exists
$\psi^*\mcl'$-effective minimal rational component $\mch$ by Proposition \ref{apa} and Proposition \ref{ap}. Let $s\in
S$ be an unramified point so that $x=f(s)\in X$ does not lie in $f(R)$ where $R$ is the ramification divisor. Let
$\mcc_x$ be the variety of minimal rational tangents associated to $\mch$ at $x$. By Proposition 3 of \cite{hm2},
there exists a variety $\mck$ which parametrizes deformation of an irreducible and reduced curve on $S$ such that
$\mcd_s\doteq\kappa_s(\mck_s)=df^{-1}(\mcc_x)\subset\proj T_s(S)$ is a \textit{variety of distinguished tangents} (see
\cite{hm2} for the definition) where $\mck_s$ is the elements of $\mck$ which pass through $s$ and
$\kappa_s:\mck_s\to\proj T_s(S)$ is the tangent map. By the proof of Proposition 4 in \cite{hm2}, $\kappa_s$ is
$Q_s$-equivariant and $\mcd_s$ is $Q_s$-invariant at a very general point. Let $D=\ker{d\pi}$. We want to show that
$\varphi$ preserves an integrable distribution which is contained in the complementary integrable distribution $D^c$
of $D$. To do this, we need to prove that ${\mathcal D}_s\cap\proj D_s=\emptyset$. An element of $\mck$ is of positive
degree with respect to $f^*\psi^*\mcl'$ and thus is not tangent to $D$ at a generic point. To prove that it is not
tangent to $D$ at every point, we need to study foliations on $\proj T(S)$ defined by orbits of highest weight
vectors.

Let $\eta\in\mbg_{0,\cdots,0,1,0,\cdots,0}\subset\mbg_1$ be a highest weight vector. We have seen in (4.2) that there
is a rational curve $C_\eta$ associated with $\eta$. Let ${\mathcal F}_o$ be the set of rational curves obtained by
applying $Q$-action on $C_\eta$ and $\mcw_o\subset\proj T_o(S)$ be the highest weight orbit of $\eta$. Let $\mcw_s$
and ${\mathcal F}_s$ be the translates of $\mcw_o$ and ${\mathcal F}_o$ at $s\in S$ respectively. Then $Q_s$ acts
transitively on $\mcw_s$ and ${\mathcal F}_s$. The tangent map on ${\mathcal F}_s$ is denoted by $\tau_s:{\mathcal
F}_s\to\proj T_s(S)$. We have
\begin{lem}\label{ar}
Let $s\in S$ be a very general point. If $\mcw_s\subset\mcd_s$, then $\kappa_s^{-1}(\mcw_s)$ is
smooth and the restriction of the tangent map $\kappa_s|_{\kappa_s^{-1}(\mcw_s)}:{\kappa_s^{-1}(\mcw_s)}\to\proj
T_s(S)$ is an immersion.
\end{lem}
\begin{proof}
Suppose $\mcw_s\subset\mcd_s$. Since $\kappa_s$ is equivariant, $\kappa_s^{-1}(\mcw_s)$ is $Q_s$-invariant. Since
$\kappa_s$ is finite by Theorem 3.4 of \cite{ke1}, $Q_s$ acts transitively on each irreducible component of
$\kappa_s^{-1}(\mcw_s)$. If $\kappa_s^{-1}(\mcw_s)$ is singular at an element $C\in{\kappa_s^{-1}(\mcw_s)}$, then
$\kappa_s^{-1}(\mcw_s)$ is singular at every element of the irreducible component of $\kappa_s^{-1}(\mcw_s)$
containing $C$ which is impossible. Therefore $\kappa_s^{-1}(\mcw_s)$ is smooth. Using a similar argument,
$d(\kappa_s|_{\kappa_s^{-1}(\mcw_s)})$ is non-degenerate at every element of $\kappa_s^{-1}(\mcw_s)$ and
$\kappa_s|_{\kappa_s^{-1}(\mcw_s)}$ is an immersion.
\end{proof}
\begin{lem}\label{as}
Let $\eta\in\mcw_s$ be a highest weight vector, $\zeta\in T_s(S)$ be a vector not proportional to $\eta$. There exists
$\gamma\in Q_s$ which fixes $\eta$ and acts non-trivially on $\zeta$.
\end{lem}
\begin{proof}
Identify $T_s(S)$ with $\bigoplus\limits_{\alpha\in\Phi^+}\mbg_\alpha$ and consider the adjoint representation of
$\mbg_0$ on $T_s(S)$. It suffices to show that there exists $h\in\mbh$ such that $[h,\eta]=0$ and $[h,\zeta]\neq0$. We
may assume that $\eta=E_\beta\in\mbg_\beta$ where $\beta$ is a highest root. Then $[h,\eta]=0$ if and only if
$\beta(h)=0$. Now $\{h:[h,\eta]=0\}$ form a hyperplane in $\mbh$. For any $\zeta=\sum\limits_{\alpha\in\Phi^+}a_\alpha
E_\alpha$, $E_\alpha\in\mbg_\alpha$, since $\{E_\alpha\}_{\alpha\in\Phi^+}$ are linearly independent, $[h,\zeta]=0$
only if $\alpha(h)=0$ whenever $a_\alpha\neq0$. Now there is some $\gamma\neq\beta$ such that $a_\gamma\neq0$ since
$\zeta$ is not proportional to $\eta$. Then there exists $h\in\mbh$ such that $\beta(h)=0$ and $\gamma(h)\neq0$ since
$\gamma$ cannot be proportional to
$\beta$. Therefore $[h,\eta]=0$ and $[h,\zeta]\neq0$.
\end{proof}

The following lemma shows that if ${\mathcal D}_s$ contains ${\mathcal W}_s$ at a general point $s$, we can recover
elements in $\mck$ by the foliation defined by $\mcw$.
\begin{prop}
If ${\mathcal D}_s$ contains ${\mathcal W}_s$ at a general point $s\in S$, then $\mck$ contains ${\mathcal F}$ .
\end{prop}
\begin{proof}
The restriction of the tangent map $\kappa_s|_{\kappa_s^{-1}(\mcw_s)}:\kappa_s^{-1}(\mcw_s)\to\proj T_s(S)$ is an
immersion by Lemma \ref{ar}. If ${\mathcal D}_s=\kappa_s({\mathcal K}_s)$ contains ${\mathcal W}_s$, then
$\kappa_s|_{\kappa_s^{-1}({\mathcal W}_s)}:\kappa_s^{-1}({\mathcal W}_s)\to{\mathcal W}_s$ is a covering map. Since
$\mcw_s$ is homogeneous which must be simply connected, each connected component of $\kappa_s^{-1}({\mathcal
W}_s)\subset{\mathcal K}_s$ is isomorphic to ${\mathcal W}_s$. We can define a morphism $\Psi_s:{\mathcal
W}_s\to{\mathcal K}_s$ such that the image of $\Psi_s$ is a connected component of $\kappa_s^{-1}({\mathcal W}_s)$
and $\kappa_s\circ\Psi_s$ is the identity map. On the other hand, the tangent map $\tau_s:{\mathcal F}_s\to\mcw_s$ is
injective. We can define another map $\Phi_s:\mcw_s\to{\mathcal F}_s$ by the inverse of $\tau_s$. For any
$v\in\mcw_s$, both $\Psi_s(v)$ and $\Phi_s(v)$ can be lifted to curves on $\proj T(S)$ by the natural projection. Let
$L$ be the tautological line bundle on $\proj T(S)$. For any $t_v\in L_v$ which is a tangent of $\Psi_s(v)$ and
$\Phi_s(v)$ on $S$ at $s$, let $t_1,t_0\in T_v(\proj T(S))$ be the tangents of $\Psi_s(v)$ and $\Phi_s(v)$ at $s$
which descend to $t_v$ under the natural projection. Note that $t_1-t_0\in T_v(\proj T_s(S))$ since it projects to
zero under the natural projection. Define $\theta(v):L_v\to T_v(\proj T_s(S))$ by $\theta(v)(t_v)=t_1-t_0$, then
$\theta\in\Gamma(\mcw_s,L^*\otimes T(\proj T_s(S)))$ is a $Q_s$-invariant section. We are going to show that $\theta$
must be zero everywhere. Suppose there exists $v_0\in\mcw_s$ such that $\theta(v_0)\neq0$. Let $\eta\in L_{v_0}$ be a
non-zero vector and $0\neq\bar{\zeta}=\theta(v_0)(\eta)\in T_{v_0}(\proj T_s(S))$. Write $\bar{\zeta}=\zeta (\bmod\cpx
v_0)$. Then $\eta$ and $\zeta$ are not proportional. By Lemma \ref{as} there is an element in $Q_s$ which fixes $\eta$
and acts non-trivially on $\zeta$. This is impossible since $\theta$ is $Q_s$-invariant. Thus $\theta$ must be
identically zero. This implies that the foliation defined by ${\mathcal D}$ and ${\mathcal W}$ are the same on $\mcw$.
Therefore the integral curves of ${\mathcal D}$ are the same as ${\mathcal W}$. Thus ${\mathcal K}$ contains
${\mathcal F}$.
\end{proof}
\begin{lem}\label{av}
If elements of ${\mathcal F}$ are of degree zero with respect to $\psi^*\mcl'$, then $\mathcal D_s$ does not intersect
$\mcw_s$ for a general point $s\in S$.
\end{lem}
\begin{proof}
Suppose $\mathcal D_s$ and $\mcw_s$ have non-empty intersection at a general point $s\in S$. Since ${\mathcal
D}_s$ is $Q_s$-invariant, ${\mathcal D}_s$ must contain ${\mathcal W}_s$ and thus ${\mathcal K}$ contains ${\mathcal
F}$. This contradicts to the fact that elements of ${\mathcal K}$ are of positive degree with respect to
$\psi^*\mcl'$.
\end{proof}
\begin{prop}\label{pc}
Let $f:S\to X$ be a ramified finite surjective holomorphic map and $\varphi$ be the induced intertwining map. Suppose
there exists a commutative diagram ($\dagger$) as in Proposition \ref{aq}. Let $D=\ker(d\pi)$ and $D^c$ be the
complementary integrable distribution of $D$.
\begin{enumerate}
\item If $f$ restricted on generic fiber of $\pi$ is unramified, then $\varphi$ preserves an integrable distribution
$E$ such that $E\subset D^c$.
\item If $f'$ is bijective, then $\varphi$ preserves $D^c$.
\end{enumerate}
\end{prop}
\begin{proof}
For a generic $s\in S$, we may assume that $s$ is $o$. For any orbit of highest weight vector $\mcw_o\subset\proj
D_o$, ${\mathcal D}_o$ does not intersect $\mcw_o$ by Lemma \ref{av}. Thus ${\mathcal D}_o$ does not intersect $\proj
D_o$. Since every equivariant distribution not contained in $D^c$ will intersect $D$, ${\mathcal D}_o$ is contained in
$\proj D^c_o$. Let $E$ be the integrable distribution generated by ${\mathcal D}_o$ for generic $s$. It is obvious
that $E$ is contained in $D^c$ and is invariant under $\varphi$. This proves the first statement.

For the second statement, suppose $f'$ is bijective. If $E\neq D^c$, let $D''$ be the integrable distribution
generated by $E\cup D$, which is invariant under $\varphi$. By Proposition \ref{aq}, there are induced maps
$S\stackrel{\pi''}{\to}S''\stackrel{f''}{\to}X''$ such that $\ker(d\pi'')=D''$. Note that $f''$ is also bijective.
By the argument of the first statement, there exists a $\varphi$-invariant integrable distribution $E''$ which is
contained in the complementary integrable distribution of $D''$. Thus $E''$ is contained in the complementary
integrable distribution of $D$. The integrable distribution generated by $E\cup E''$ is $\varphi$-invariant and
contains $E$ as a proper subset. Since there are only finitely many integrable distributions, repeating this argument
we conclude that $D^c$ is $\varphi$-invariant.
\end{proof}

\section*{5. The Frobenius bracket and Chern numbers}

In this section, we are going to prove Proposition \ref{aj} which says that if $\varphi$ preserves a proper
equivariant distribution $D$, then $\varphi$ preserves either $D^1$ or a proper integrable distribution. This
proposition is an analogue of Proposition 10 in \cite{hm2} except that in the situation of \cite{hm2}, proper
integrable distributions do not arise since $S$ is of Picard number one. When $S$ is of Picard number one, there is a
filtration by equivariant distributions $D^1\subset D^2\subset\cdots\subset D^m=T(S)$. Any proper equivariant
distribution $D$ is a $D^k$ for some $k<m$. To prove Proposition 10 of \cite{hm2}, Hwang and Mok showed that $D^1$ can
be recovered by $D^k$. First they made the observation that any $P$-invariant subset ($S=G/P$) which is not contained
in $D^1$ must contain the highest weight orbit $\mcw^2\subset \proj D^2$ (Proposition 5, \cite{hm2}). Since the rank
of $F_\xi$ is constant along the $P$-orbit of $\xi$ and is lower-semi-continuous on $\xi\in D_o$,
Proposition 10 of \cite{hm2} is reduced to the inequality $\rank F_{\eta_1}<\rank F_{\eta_2}$, where $\eta_i$, $i=1,2$,
is the highest weight vector in $D^i$. The required inequality is obtained from the relationship between the rank
of $F_\xi$ and the Chern number of $T(S)/D^i$ restricted on the rational curve associated with highest
weight vector. When $S$ is of Picard number larger than 1, we need to study the more complicated multi-filtration by
equivariant distributions on $S$.\\

(5.1) By considering the $({\mathbb C}^*)^l$-action corresponding to ${\mathfrak z}$ on ${\mathfrak q}$, where ${\mathfrak
z}$ is the center of ${\mathfrak l}$ and ${\mathfrak l}$ is the reductive Lie algebra corresponding to a Levi
decomposition ${\mathfrak q}={\mathfrak u}+{\mathfrak l}$, we have the following analogue of Proposition 5 of
\cite{hm2} with a similar proof.
\begin{prop}\label{ps}
Choose an identification $T_o(S)=\bigoplus\limits_{\alpha\in\Phi^+}\mbg_\alpha$.
\begin{enumerate}
\item For any $v\in\proj T_o(S)\setminus\proj\left(\bigoplus\limits_{1\leq i\leq
l}\mbg_{0,\dots,0,m_i,0,\dots,0}\right)$, the closure of the $Q$-orbit of $v$ contains a highest weight orbit which is
of the form $W^{0,\cdots,0,1,0,\cdots,0,1,0,\cdots,0}$.
\item For any $v\in
\proj\mbg_{0,\cdots,0,1,0,\cdots,0,1,0,\cdots,0}$ which does not lie in\\$\proj\mbg_{0,\cdots,0,1,0,\cdots,0,0,0,\cdots,0}\cup\proj\mbg_{0,\cdots,0,0,0,\cdots,0,1,0,\cdots,0}$,
the closure of the $Q$-orbit of $v$ contains $W^{0,\cdots,0,1,0,\cdots,0,0,0,\cdots,0}\cup
W^{0,\cdots,0,0,0,\cdots,0,1,0,\cdots,0}$. \item For any $v\in \proj\left(\bigoplus\limits_{1\leq i\leq
l}\mbg_{0,\dots,0,m_i,0,\dots,0}\right)\setminus\proj D^1_o$, the closure of the $Q$-orbit of $v$ contains a highest
weight orbit which is of the form $W^{0,\cdots,0,2,0,\cdots,0}$. \item For any $v\in
\proj\mbg_{0,\cdots,0,2,0,\cdots,0}\setminus\proj\mbg_{0,\cdots,0,1,0,\cdots,0}$, the closure of the $Q$-orbit of $v$
contains $W^{0,\cdots,0,2,0,\cdots,0}$.
\end{enumerate}
\end{prop}
We now describe the Frobenius bracket associated with a proper distribution $D$. For any distribution $D\subset T(S)$ and
$\eta,\xi\in D_o=D\cap T_o(S)$, extend $\eta$ and $\xi$ to $D$-valued vector fields $\tilde{\eta}$ and $\tilde{\xi}$
on a neighborhood of $o$ such that $\tilde{\eta}(o)=\eta$ and $\tilde{\xi}(o)=\xi$. Then the Frobenius bracket
$[\tilde{\eta},\tilde{\xi}]_o(\bmod D_o)$ at $o$ does not depend on the extension. Fix $\eta\in D_o$, we define the
following map
$$F^D_\eta:D_o\to T_o(S)/D_o$$
$$F^D_\eta(\xi)=[\tilde{\eta},\tilde{\xi}]_o\ (\bmod D_o).$$
The Frobenius bracket agrees with the Lie bracket of left invariant vector fields $\mbg$ on $G$ and we have
\begin{lem}
Let $\rho:G\to S=G/Q$ be the natural projection and $D$ be an equivariant distribution on $S$. Let
$\tilde{\eta},\tilde{\xi}\in\mbg$ be left invariant vector fields on $G$ such that
$\eta=d\rho(\tilde{\eta}_e),\xi=d\rho(\tilde{\xi}_e)\in D_o$. Then
$$F^D_\eta(\xi)=d\rho([\tilde{\eta},\tilde{\xi}]_e)\ (\bmod D_o)$$
where $[\cdot,\cdot]$ is the Lie bracket on $\mbg$.
\end{lem}

Now, we relate the rank of $F^D_\eta$ with the Chern number of $T(S)/D$ with respect to rational curve
associates to positive root.
\begin{lem}\label{ld}
For each highest root $\alpha\in\Phi_{k_1,\cdots,k_l}$, the rational curve $C_\alpha$ has degree
$k_i\frac{\langle\alpha_{r_i},\alpha_{r_i}\rangle}{\langle\alpha,\alpha\rangle}$ with respect to ${\mathcal L}_i$. In
particular if $\alpha_j\in\mbg_{0,\cdots,0,1,0,\cdots,0}$, $j$-th entry is 1, then $\mcl_i\cdot
C_{\alpha_j}=\delta_{ij}$.
\end{lem}
\begin{proof}
The coroot of $\alpha=\sum s_j\alpha_j\in\Phi_{k_1,\cdots,k_l}$, with $s_{r_i}=k_i$, is $H_\alpha=\sum
s_j\frac{\langle\alpha_j,\alpha_j\rangle}{\langle\alpha,\alpha\rangle}H_{\alpha_j}$. Therefore the $r_i$-th
fundamental weight defining ${\mathcal L}_i$ has value
$s_i\frac{\langle\alpha_{r_i},\alpha_{r_i}\rangle}{\langle\alpha,\alpha\rangle}$ on $H_\alpha$.
\end{proof}

Let $D$ be a proper equivariant distribution. It is known from Proposition \ref{pk} that
$D=\sum\limits_{\lambda\in\Lambda}D^\lambda$ for some finite set $\Lambda\subset{\mathbb Z}^l$ with $\lambda>0$ for all
$\lambda\in\Lambda$. Let $\alpha\in\Phi_{k_1,\cdots,k_l}$ be a highest root and $\eta=\eta_\alpha\in\mbg_\alpha$ be a
highest weight vector such that $\eta_\alpha\in D_o$. Let $F^D_\eta:D_o\to T_o(S)/D_o$ be a map defined by $F^D_\eta(\xi)=[\eta,\xi](\bmod D_o)$. To
calculate the first Chern number of $T(S)/D$ restricted to $C_\alpha$, define inductively a sequence
$$\cdots\stackrel{F^{D^{-2}_\eta}_\eta}{\to}D^{-1}_\eta/{D^{-2}_\eta}\stackrel{F^{D^{-1}_\eta}_\eta}{\to}D^0_\eta/{D^{-1}_\eta}
\stackrel{F^{D^0_\eta}_\eta}{\to}D^1_\eta/{D^0_\eta}\stackrel{F^{D^1_\eta}_\eta}{\to}D^2_\eta/{D^1_\eta}\stackrel{F^{D^2_\eta}_\eta}{\to}\cdots,$$
where
$$D^0_\eta=D,$$
$$D^k_\eta=\ker(F^{D^{k+1}_\eta}_\eta),\text{ for }k<0,$$
$$D^k_\eta=\pi_{k-1}^{-1}(\text{Im}(F^{D^{k-1}_\eta}_\eta)),\text{ for }k>0.$$
Here $\pi_i:T(S)\to T(S)/D^i_\eta$ denotes the natural projection. On the other hand, denote
$$\bar{\Lambda}_\alpha^k=\{\lambda\in{\mathbb Z}^l:(0,\cdots,0)<\lambda\le\xi+k\alpha,\text{ for some }\xi\in\Lambda\},$$
$$\mbg_\alpha^k=\bigoplus\limits_{\lambda\in\bar{\Lambda}_\alpha^k}\mbg_\lambda,$$
for $k\in{\mathbb Z}$, where $\alpha$ is identified with an element in ${\mathbb Z}^l$. We may define
$$F^k_\eta:\mbg_\alpha^k/\mbg_\alpha^{k-1}\to\mbg_\alpha^{k+1}/\mbg_\alpha^k,$$
$$F^k_\eta(\xi)=ad\:\eta(\xi)=[\eta,\xi]\ (\bmod\mbg_\alpha^k).$$
Since the Frobenius bracket on $D^k_\eta$ agrees with the Lie bracket on $\mbg_\alpha^k$, we have
$$(D^k_\eta)_o=(\mbg_\alpha^k+\mbg^\perp_\alpha)/\mbq,\text{ for }k<0,$$
$$(D^k_\eta)_o=\mbg^k_\alpha/\mbq,\text{ for }k\ge0.$$
where $\mbg^\perp_\alpha=\{\xi\in\mbg^0_\alpha:[\eta,\xi]=0\}$. Thus
$$(D^k_\eta/D^{k-1}_\eta)_o=\mbg^k_\alpha/\mbg^{k-1}_\alpha,\ k\neq0,$$
$$(D^0_\eta/D^{-1}_\eta)_o=\mbg^0_\alpha/(\mbg^{-1}_\alpha+\mbg^\perp_\alpha)$$
and we have
$$F^{D^k_\eta}_\eta=F^k_\eta.$$
Now, define
$$F^{(k)}_\eta:\mbg^0_\alpha/\mbq\to\mbg/\mbg^0_\alpha,$$
$$F^{(k)}_\eta(\xi)=(ad\:\eta)^k\xi\ (\bmod\mbg_\alpha^0).$$
Consider the decomposition
$$\mbg^0_\alpha=\bigoplus\limits_{k\leq0}\mbg^k_\alpha/\mbg^{k-1}_\alpha,$$
$$\text{Im}(F^{(k)}_\eta)=\bigoplus\limits_{s\leq k}\mbg^s_\alpha/\mbg^{s-1}_\alpha\subset\mbg/\mbg^0_\alpha,$$
we have
$$F^{(1)}_\eta=F^1_\eta,$$
$$F^{(2)}_\eta=F^0_\eta\circ F^{-1}_\eta+F^1_\eta\circ F^0_\eta,$$
$$F^{(3)}_\eta=F^0_\eta\circ F^{-1}_\eta\circ F^{-2}_\eta+F^1_\eta\circ F^0_\eta\circ F^{-1}_\eta+F^2_\eta\circ F^1_\eta\circ F^0_\eta.$$
Therefore
$$\rank(F^{(1)}_\eta)=\rank(F^{D^0_\eta}_\eta)$$
$$\rank(F^{(2)}_\eta)=\rank(F^{D^0_\eta}_\eta\circ F^{D^{-1}_\eta}_\eta)+\rank(F^{D^1_\eta}_\eta\circ F^{D^0_\eta}_\eta)$$
$$\rank(F^{(3)}_\eta)=\sum\limits_{i=0}^2\rank(F^{D^i_\eta}_\eta\circ F^{D^{i-1}_\eta}_\eta\circ F^{D^{i-2}_\eta}_\eta)$$
and so on.

The following Lemma is an extension of Lemma 8 in \cite{hm2} with the assumption on length of roots removed.
\begin{lem}\label{la}
For each highest root $\alpha\in\Phi_{k_1,\cdots,k_l}$ and highest weight vector $\eta=\eta_\alpha\in\mbg_\alpha$, the
first Chern number of $T(S)/D$ restricted to $C_\alpha$ is
$$\sum\limits_{k=1}^{3}\rank(F^{(k)}_\eta).$$
\end{lem}
\begin{proof}
By Grothendieck's splitting, the Chern number of $T(S)/D$ restricted to $C_\alpha$ is
$\sum\limits_{\beta\in\Phi^+\setminus\Phi_{\bar{\Lambda}}}\beta(H_\alpha)$, where
$\Phi_{\bar{\Lambda}}=\{\gamma\in\Phi^+:\gamma\in\mbg_\lambda,\ \lambda\in\bar{\Lambda}\}$. Now $\beta(H_\alpha)=p-q$
where $\beta-p\alpha,\cdots,\beta+q\alpha$ is the maximal $\alpha$-string through $\beta$. By the structure of simple
Lie algebra, the length of a maximal string depends on the ratio of the lengths of $\beta$, $\alpha$ and is at most 4.
We have \bea
&&\sum\limits_{\beta\in\Phi^+\setminus\Phi_{\bar{\Lambda}}}\beta(H_\alpha)\\
&=&\sum\limits^3_{k=1}\left|\{\beta\in\Phi^+\setminus\Phi_{\bar{\Lambda}}:\beta-k\alpha\in\Phi\}\right|
-\sum\limits^3_{k=1}\left|\{\beta\in\Phi^+\setminus\Phi_{\bar{\Lambda}}:\beta+k\alpha\in\Phi\}\right|\\
&=&\sum\limits^3_{k=1}\left|\{\beta\in\Phi^+\setminus\Phi_{\bar{\Lambda}}:\beta-k\alpha\in\Phi_{\bar{\Lambda}}\}\right|\\
&=&\sum\limits^3_{k=1}\left|\{\gamma\in\Phi_{\bar{\Lambda}}:\gamma+k\alpha\in\Phi^+\setminus\Phi_{\bar{\Lambda}}\}\right|\\
&=&\sum\limits^3_{k=1}\rank(F^{(k)}_\eta).\eea
\end{proof}

(5.2) For any distribution $D$ of $S$, define the Cauchy characteristic of $D$ by $Ch(D)\doteq\{\eta\in
D:[\eta,D]=0(\bmod D)\}$, where $[\cdot,\cdot]$ is the Frobenius bracket with respect to $D$. In other words, $Ch(D)$ consists of all vectors $\eta\in D$ with $\rank(F^D_\eta)=0$. From the Jacobi identity, $Ch(D)$ is always integrable for any $D$ if it is non-empty. Now we can prove the following main result of this section.
\begin{prop}\label{aj}
If $\varphi$ preserves a proper equivariant distribution, then one of the followings holds.
\begin{enumerate}
\item $\varphi$ preserves $D^1$.
\item $\varphi$ preserves a proper equivariant integrable distribution.
\end{enumerate}
\end{prop}
\textit{Proof.} Suppose $\varphi$ preserves a proper equivariant distribution $D$. Consider $d\varphi$ as an element in $\proj
GL(T_o(S))$. For any $k=1,2,3$ and $\eta\in D_o$, $\rank(F^{(k)}_\eta)$ is kept unchanged under $d\varphi$ and is
constant along the $Q$-orbit of $\eta$. Thus $\varphi$ preserves the set of vectors $\eta\neq0$ for which
$\sum\limits^3_{k=1}\rank(F^{(k)}_\eta)$ is minimum. In particular $\varphi$ preserves the Cauchy characteristic $Ch(D)$ of $D$. Now if $D$ does not contain $D^1$, then $\varphi$ preserves the equivariant integrable distribution
generated by $D$ which is a proper integrable distribution. If $\varphi$ contains $D^1$ and $\rank(F^D_\eta)=0$ for
some highest weight vector $\eta\in D^1_o$, then $\varphi$ preserves $Ch(D)\neq\{0\}$ which is a proper equivariant
integrable distribution.

We claim that $\varphi$ preserves a proper equivariant
distribution which is contained in the sum of all minimal
integrable distributions $\bigoplus\limits_{1\leq i\leq
l}D^{0,\dots,0,m_i,0,\dots,0}$. If $D$ contains an equivariant
distribution of the form $D^{0,\dots,0,1,0,\dots,0,1,0,\dots,0}$,
let $\eta_{11}\in\mbg_{0,\dots,0,1,0,\dots,0,1,0,\dots,0}$ be highest weight
vector such that $\sum\limits^3_{k=1}\rank(F^{(k)}_{\eta_{11}})$ is minimum among
all highest weight orbit of the form $W^{0,\dots,0,1,0,\dots,0,1,0,\dots,0}$ which is contained in $D_o$. Let $\eta_{10}$, $\eta_{01}$ be highest weight
vectors in the corresponding $\mbg_{0,\dots,0,1,0,\dots,0,0,0,\dots,0}$,
$\mbg_{0,\dots,0,0,0,\dots,0,1,0,\dots,0}$ respectively. By Lemma \ref{ld}, the degree of $C_{\eta_{11}}$ are larger than
or equal to the degree of $C_{\eta_{10}}+C_{\eta_{01}}$ with
respect to $\mcl_t$ for all $1\leq t\leq l$. Therefore
$$c(T(S)/D)\cdot C_{\alpha_{10}}+c(T(S)/D)\cdot C_{\alpha_{01}}\leq c(T(S)/D)\cdot C_{\alpha_{11}}.$$
This equality along with Lemma \ref{la} gives
$$\sum\limits^3_{k=1}\rank(F^{(k)}_{\eta_{10}})+\sum\limits^3_{k=1}\rank(F^{(k)}_{\eta_{01}})\leq\sum\limits^3_{k=1}\rank(F^{(k)}_{\eta_{11}}).$$
Since both terms on the left side of the above inequality are positive, both of them are strictly smaller than the
right side. Since $\sum\limits^3_{k=1}\rank(F^{(k)}_\eta)$ is lower semi-continuous with respect to $\eta$, the set
$${\mathcal S}=\{\eta\in D_o:\sum\limits^3_{k=1}\rank(F^{(k)}_\eta)<\sum\limits^3_{k=1}\rank(F^{(k)}_{\eta_{11}})\}\subset T_o(S)$$
is a closed equivariant subset preserved by $\varphi$ which contains $\eta_{10},\eta_{01}$ and does not contain
$\eta_{11}$. By the construction of $\eta_{11}$, ${\mathcal S}$ is disjoint from any highest weight orbit of the form
$W^{0,\dots,0,1,0,\dots,0,1,0,\dots,0}$. By Proposition \ref{ps}, ${\mathcal S}$ and thus its linear span is contained
in $\bigoplus\limits_{1\leq i\leq l}D^{0,\dots,0,m_i,0,\dots,0}$ and is preserved by $\varphi$. For the rest of the
proof, we replace $D$ by this distribution.

We will finish the proof of Proposition \ref{aj} case by case.
Observe that if $\varphi$ preserves a proper distribution of
$D^1$, then $\varphi$ preserves the proper integrable equivariant
distribution generated by this distribution. Therefore it suffices
to show that $\varphi$ preserves an equivariant distribution which
is contained in $D^1$. If $D$ contains an equivariant distribution
of the form $D^{0,\dots,0,2,0,\dots,0}$, let $\eta_1$, $\eta_2$ be
highest weight vectors in $\mbg_{0,\dots,0,1,0,\dots,0}$ and
$\mbg_{0,\dots,0,2,0,\dots,0}$ respectively. Using the same
argument above, we can conclude that $\varphi$ preserves an
equivariant subdistribution of $D$ which contains
$D^{0,\dots,0,1,0,\dots,0}$ and does not contain
$D^{0,\dots,0,2,0,\dots,0}$ by establishing the inequality
$$\rank(F^D_{\eta_1})<\rank(F^D_{\eta_2}).\eqno{(*)}$$
Then it follows readily that $\varphi$ preserves a distribution which is contained in $D^1$.

First we consider the case when all roots in $\Delta_1$ are of maximal length. This
includes the cases when
\begin{enumerate}
\item $\mbg$ is of type $A$, $D$ or $E$, i.e. when all roots are of equal length, or
\item $\mbg$ is of type $B_k$ and $\Delta_1\subset\{\alpha_1,\cdots,\alpha_{k-1}\}$, or
\item $\mbg$ is of type $C_k$ and $\Delta_1=\{\alpha_k\}$, or
\item $\mbg$ is of type $F_4$ and $\Delta_1\subset\{\alpha_3,\alpha_4\}$, or
\item $\mbg$ is of type $G_2$ and $\Delta_1=\{\alpha_2\}$,
\end{enumerate}
where the numbering of roots in \cite{ti} is being used. In Case 3 or Case 5, $S$ is of Picard number 1 which has been
studied in \cite{hm2}.
\begin{lem}\label{pa}
Suppose all simple roots in $\Delta_1$ are maximal, then $\varphi$ preserves an equivariant distribution which is contained in $D^1$.
\end{lem}
\begin{proof}
Let $\eta_1$ and $\eta_2$ be highest weight vectors in $\mbg_{0,\dots,0,1,0,\dots,0}$ and\\
$\mbg_{0,\dots,0,2,0,\dots,0}$. The roots correspond to $\eta_1$ and $\eta_2$ are also maximal as all simple roots in
$\Delta_1$ are maximal. By Lemma \ref{ld}, $c(T(S)/D)\cdot C_{\eta_2}=2c(T(S)/D)\cdot C_{\eta_1}$ which implies
$$\rank(F^D_{\eta_2})=2\,\rank(F^D_{\eta_1})$$
and the desired inequality $(*)$ is obtained. We remark that for $t=1,2$,
$\sum\limits^3_{k=1}\rank(F^{(k)}_{\eta_t})=\rank(F^D_{\eta_t}$) since $\rank(F^{(2)}_{\eta_t})=\rank(F^{(3)}_{\eta_t})=0$ in this case.
\end{proof}

Next we consider the case when $\Delta_1$ contains a maximal root. This includes the following four cases.
\begin{enumerate}
\item $\mbg$ is of type $B_k$ and $\alpha_i\in\Delta_1$ for some $1\leq i\leq k-1$, or
\item $\mbg$ is of type $C_k$ and $\alpha_k\in\Delta_1$, or
\item $\mbg$ is of type $F_4$ and $\Delta_1$ contains $\alpha_3$ or $\alpha_4$, or
\item $\mbg$ is of type $G_2$ and $\alpha_2\in\Delta_1$.
\end{enumerate}
\begin{lem}\label{lk}
If $\Delta_1$ contains a maximal root, then $\varphi$ preserves an equivariant distribution which is contained in $D^1$.
\end{lem}
\begin{proof}
We only need to consider the cases when some roots in $\Delta_1$ are not maximal. In each of these cases, there is
only one node of depth $m_i>1$. Let $\Delta_1=\{\alpha_{r_1}, \alpha_{r_2}, \cdots, \alpha_{r_l}\}$.\\
For $B_k$:\\
In this case $\alpha_{r_l}=\alpha_k\in\Delta_1$, otherwise all roots are maximal. The only node of depth larger than 1
is $\alpha_{r_l}=\alpha_k$. We have $m_1=m_2=\cdots=m_{l-1}=1$ and $m_l=2$. By direct computation
$$\rank(F^D_{\eta_{0,\dots,0,1}})=r_{l-1}-r_{l-2}\qquad\text{and}\qquad\rank(F^D_{\eta_{0,\dots,0,2}})=2(r_{l-1}-r_{l-2}).$$
For $C_k$ and $G_2$:\\
In this case $\alpha_{r_l}=\alpha_k\in\Delta_1$, otherwise $\Delta_1$ does not contain a maximal root. We have
$m_1=m_2=\cdots= m_l=1$. There is no node of depth larger than 1 and nothing is needed to prove.\\
For $F_4$:\\
If $\Delta_1$ contains $\alpha_3$, then all $m_i$ are 1 and there is nothing to prove. When $\alpha_{r_l}=\alpha_4$,
we have the following three possible cases.\\
Case I. $\Delta_1=\{\alpha_1,\alpha_4\}$.\\
In this case $m_1=2$ and $m_2=1$, $\eta_1=\eta_{1,0}=1210$, $\eta_2=\eta_{2,0}=2210$.
$$\Phi_{\bar{\Lambda}}=\{0001,1000,0011,1100,0111,1110,0211,1210,0221,2210\},$$
$$\Phi_{{\bar{\Lambda}}^1_{\eta_1}}\setminus\Phi_{\bar{\Lambda}}=\{1111,1211,1221,1321\},$$
$$\Phi_{{\bar{\Lambda}}^1_{\eta_2}}\setminus\Phi_{{\bar{\Lambda}}^1_{\eta_1}}=\{2211,2221,2321,2421,2431\}.$$
$$\rank(F^D_{\eta_1})=3\qquad\text{and}\qquad\rank(F^D_{\eta_2})=5.$$
Case II. $\Delta_1=\{\alpha_2,\alpha_4\}$.\\
In this case $m_1=2$ and $m_2=1$, $\eta_1=\eta_{1,0}=1110$, $\eta_2=\eta_{2,0}=2210$.
$$\Phi_{\bar{\Lambda}}=\{0001,0100,0011,0110,1100,0210,1110,1210,2210\}$$
$$\Phi_{{\bar{\Lambda}}^1_{\eta_1}}\setminus\Phi_{\bar{\Lambda}}=\{0111,1111\}$$
$$\Phi_{{\bar{\Lambda}}^1_{\eta_2}}\setminus\Phi_{{\bar{\Lambda}}^1_{\eta_1}}=\{0211,0221,1211,1221,2211,2221\}$$
$$\rank(F^D_{\eta_1})=1\qquad\text{and}\qquad\rank(F^D_{\eta_2})=2.$$
Case III. $\Delta_1=\{\alpha_1,\alpha_2,\alpha_4\}$.\\
In this case $m_1=m_3=1$ and $m_2=2$, $\eta_1=\eta_{0,1,0}=0110$, $\eta_2=\eta_{0,2,0}=0210$.
$$\Phi_{\bar{\Lambda}}=\{0001,0100,1000,0011,0110,0210\}$$
$$\Phi_{{\bar{\Lambda}}^1_{\eta_1}}\setminus\Phi_{\bar{\Lambda}}=\{1100,0111,1110\}$$
$$\Phi_{{\bar{\Lambda}}^1_{\eta_2}}\setminus\Phi_{{\bar{\Lambda}}^1_{\eta_1}}=\{0211,1210,0221\}$$
$$\rank(F^D_{\eta_1})=2\qquad\text{and}\qquad\rank(F^D_{\eta_2})=3.$$
The desired inequalities $(*)$ are established and Lemma \ref{lk} holds in all cases.
\end{proof}

It remains to consider the case when all roots in $\Delta_1$ are
short roots. This includes the following four cases.
\begin{enumerate}
\item $\mbg$ is of type $B_k$ and $\Delta_1=\{\alpha_k\}$, or
\item $\mbg$ is of type $C_k$ and $\Delta_1\subset\{\alpha_1,\cdots,\alpha_{k-1}\}$, or
\item $\mbg$ is of type $F_4$ and $\Delta_1\subset\{\alpha_1,\alpha_2\}$, or
\item $\mbg$ is of type $G_2$ and $\Delta_1=\{\alpha_1\}$.
\end{enumerate}
In Case 1, Case 4 or Case 3 with
$\Delta_1\neq\{\alpha_1,\alpha_2\}$, $S$ is of Picard number 1. So
we only need to study Case 2 and Case 3 with
$\Delta_1=\{\alpha_1,\alpha_2\}$. We will prove that $\varphi$
preserves one of the minimal integrable distributions. The proof
of Proposition \ref{aj} is completed by the following lemma.
\begin{lem}
Suppose $\Delta_1$ does not contain maximal root, then $\varphi$ preserves one of the minimal integrable
distributions.
\end{lem}
\begin{proof}
We may assume that $D=\bigoplus\limits_{1\leq i\leq l}D^{0,\dots,0,m_i,0,\dots,0}$. Otherwise $D$ is already $D^1$
or a proper equivariant integrable distribution since in the cases we are considering, one has $m_1=\cdots=m_{r-1}=1$,
$m_r=2$ and
$$\bigoplus\limits_{1\leq i\leq
l}D^{0,\dots,0,m_i,0,\dots,0}=D^{1,0,\dots,0}\oplus D^{0,1,0,\dots,0}\oplus\cdots\oplus D^{0,\dots,0,1,0}\oplus
D^{0,\dots,0,2}.$$ When $S$ is of type $C_k$, let $E^1=D$ and $E^j=[D,E^{j-1}]$, for $j>1$. It is obvious that
$\varphi$ preserves $E^j$ for all $j>0$. Now
$$E^l=D^{1,\dots,1,2}+D^{0,1,\dots,1,2,2}+D^{0,0,1,\dots,1,2,2,2}+\cdots$$
and
$$E^{l+1}=D^{1,\dots,1,2,2}+D^{0,1,\dots,1,2,2,2}+D^{0,0,1,\dots,1,2,2,2,2}+\cdots.$$
Let $\eta_i$, $1\leq i\leq l$, be highest weight vector corresponding to node $\alpha_{r_i}$. Then
$\rank(F^{E^l}_{\eta_l})=0$ and $\rank(F^{E^l}_{\eta_i})\neq0$ for $i\neq l$. Therefore $\varphi$ preserves the
minimal integrable distribution $D^{0,\dots,0,2}$.

When $S$ is of type ($F_4,\{\alpha_1,\alpha_2\}$), let $\eta_{10}$ and $\eta_{01}$ be highest weight vectors in
$D^{1,0}$ and $D^{0,1}$ respectively. Then $\rank(F^D_{\eta_{10}})=6$ and $\rank(F^D_{\eta_{01}})=1$. Therefore
$\varphi$ preserves the minimal integrable distribution $D^{0,2}$.
\end{proof}

\section*{6. Integrable distributions}
The aim of this section is to show that the first case of Proposition \ref{aj} can be reduced to the second case. We will assume that $\varphi$ preserves $D^1$ in this section.\\

(6.1) We first consider the case when $S$ is neither
of type $(A_k,\{\alpha_1,\alpha_i\})$ $(1<i\leq k)$ nor $(C_k,\{\alpha_1,\alpha_k\})$. In this case, the following theorem of Tanaka and Yamaguchi concerning differential systems can be applied.
\begin{prop}\cite{ya}\cite{ta}\label{ac}
Let $S$ be a rational homogenous space and $U\subset S$ be a connected open set. Then a holomorphic vector field on $U$
can be extended to a global holomorphic vector field on $S$ if it preserves $D^1|_U$ except for the following cases.
\begin{enumerate}
\item $S$ is of depth 1, i.e., $S$ is a Hermitian symmetric space, \item $S$ is a contact space, \item $S$ is of type
$(A_k,\{\alpha_1,\alpha_i\})$ $(1<i<k)$ or $(C_k,\{\alpha_1,\alpha_k\})$.
\end{enumerate}
Furthermore if $S$ is a Hermitian symmetric space or a contact space, then a holomorphic vector field on $U$ can be
extended to a global holomorphic vector field on $S$ if it preserves ${\mathcal W}^1$.
\end{prop}
Theorem 5.2 of \cite{ya} says that a local holomorphic vector field on $S$ preserving $D^1$ can be extended to a global holomorphic vector field on $S$ provided that $S$ is not of type in the above list. We state the theorem in the above form in order to apply it more directly. This theorem contradicts with the following proposition and therefore $\varphi$ cannot preserve $D^1$ in this case.
\begin{prop}\label{pb}
Suppose $S$ is of Picard number $l>1$ and $f$ is ramified, then $\varphi$ cannot be extended to a biholomorphic
automorphism of $S$.
\end{prop}
\begin{proof}
Suppose the ramification divisor $R$ of $f$ is not empty and $\varphi$ can be extended to a biholomorphic
automorphism. From \cite{hm2} Proposition 14, there exists a finite subgroup $F\subset Aut(S)$ which fixes $R$
pointwise. Let $\gamma\in F$ be a non-identity element and $p\in R$ be a point fixed by $\gamma$. $\gamma$ acts on
$T_p(S)\cong\mbg/\mbq$. Since $\gamma$ fixes every vector in $T_p(R)\subset T_p(S)$ which is of codimension 1,
$T_p(S)=\cpx\eta\oplus T_p(R)$ where $\eta\in T_p(S)$ is an eigenvector of $\gamma$ with eigenvalue $\lambda\neq0,1$.
$\gamma$ induces an automorphism of $\mbg$ preserving $\mbq$. Since $\gamma(D^1)$ generates $\mbg/\mbq$ and is a
$Q$-invariant linear subspace, $\gamma(D^1)$ contains $D^1$. Thus $\gamma(D^1)=D^1$. For any $k\in{\mathbb Z}^+$,
$\gamma([\xi_1,\xi_k])=[\gamma(\xi_1),\gamma(\xi_k)]\bmod D^k$ for each $\xi_1\in D^1$, $\xi_k\in D^k$. Hence $D^k$ are
invariant under $\gamma$ and $\gamma$ acts on $D^{k+1}/D^k$. Observe that $\mbg_1\oplus\mbq$ generate $\mbg$ and
$\gamma$ acts non-trivially on $D^1$. Otherwise $\gamma$ acts trivially on all of $D^{k+1}/D^k$ which is impossible.
Therefore $\eta\in D^1$. $D^2/D^1\neq0$ since $S$ is not a symmetric space. There exists $\xi\in D^1$ such that
$[\eta,\xi]\neq0\bmod D^1$ and one may choose such $\xi$ so that $\xi\in T_p(R)$. Now
$\gamma([\eta,\xi])=[\gamma(\eta),\gamma(\xi)]=\lambda[\eta,\xi]\bmod D^1$, contradicting
the fact that eigenspace associated to $\lambda$ is one dimensional.
\end{proof}

(6.2) If $\varphi$ preserves $D^1$ and $S$ is of type $(A_k,\{\alpha_1,\alpha_i\})$ $(1<i\leq k)$ or
$(C_k,\{\alpha_1,\alpha_k\})$, we show that $\varphi$ preserves one of the minimal integrable distributions. Note that $D^{0,1}$ is integrable since $m_2=1$ in this case. When $i\neq k$, we have
\begin{lem}
If $S$ is of type ($A_k,\{\alpha_1,\alpha_i\}$), $1<i<k$, or ($C_k,\{\alpha_1,\alpha_k\}$) and $\varphi$ preserves
$D^1$, then $\varphi$ preserves $D^{0,1}$.
\end{lem}
\begin{proof}
Let $\eta_{10}$ and $\eta_{01}$ be highest weight vectors in $D^{1,0}$ and $D^{0,1}$.\\
If $S$ is of type ($A_k,\{\alpha_1,\alpha_i\}$), we have
$$\rank(F^{D^1}_{\eta_{10}})=k-i+1>1\qquad\text{and}\qquad\rank(F^{D^1}_{\eta_{01}})=1.$$
If $S$ is of type ($C_k,\{\alpha_1,\alpha_k\}$), let $E=[D^1,D^1]=D^{1,1}$. We have
$$\rank(F^E_{\eta_{10}})=1\qquad\text{and}\qquad\rank(F^E_{\eta_{01}})=0.$$
Using the same argument as above, $\varphi$ preserves $D^{0,1}$.
\end{proof}

When $S$ is of type ($A_k,\{\alpha_1,\alpha_k\}$), i.e., $S$ is a contact space of Picard number 2, there are two proper integrable distributions $D^{1,0}$ and $D^{0,1}$. To prove that $\varphi$ preserves one of them, we define a $Q$-invariant
subvariety ${\mathcal I}\subset{\mathbb P}T(S)$ and show that any $Q$-invariant subvariety not contained in $\proj
D^{1,0}\cup\proj D^{0,1}$ must contain ${\mathcal I}$. Now $S$ is biholomorphic to $\proj(T(\proj
V))=\{(X,H)\in\proj V\times\proj V^*: X\in H\}$, where $V$ is a $k+1$-dimensional vector space over $\cpx$. We have
$G=GL(V)$ and
$$Q_{(X,H)}=\{A\in GL(V): X,H\text{ are invariant under }A\}$$
acts on
$$T_{(X,H)}(S)=\{\zeta\oplus\eta\in Hom(X,V/X)\oplus Hom(H,V/H): \pi\circ\zeta=\eta|_X\},$$
where $\pi:V/X \to V/H$ is the natural projection, by $A\cdot\zeta\oplus\eta=A\zeta A^{-1}\oplus A\eta A^{-1}$. The
two minimal integrable distributions are
$$D^{1,0}_{(X,H)}=\{\zeta\in Hom(X,V/X):\pi\circ\zeta=0\}=Hom(X,H/X)$$
$$D^{0,1}_{(X,H)}=\{\eta\in Hom(H,V/H):\eta|_X=0\}=Hom(H/X,V/H).$$
Therefore $D^1_{(X,H)}=D^{1,0}_{(X,H)}\oplus D^{0,1}_{(X,H)}=Hom(X,H/X)\oplus Hom(H/X,V/H)$. Define
$${\mathcal I}_{(X,H)}=\proj\{\zeta\oplus\eta\in D^{1,0}_{(X,H)}\oplus D^{0,1}_{(X,H)}:\eta\circ\zeta=0\}.$$
Then ${\mathcal I}_{(X,H)}\subset\proj D^1_{(X,H)}\subset\proj T_{(X,H)}(S)$ is obviously $Q_{(X,H)}$-invariant.
\begin{prop}\label{at}
Any $Q_{(X,H)}$-invariant closed subvariety ${\mathcal J}\subset \proj D^1_{(X,H)}$ which is not contained in $\proj
D^{1,0}_{(X,H)}\cup\proj D^{0,1}_{(X,H)}$ contains ${\mathcal I}_{(X,H)}$.
\end{prop}
\begin{proof}
Since ${\mathcal J}$ is not contained in $D^{1,0}_{(X,H)}\cup D^{0,1}_{(X,H)}$, ${\mathcal J}$ contains an element
of the form $[\zeta\oplus\eta]$ with $\zeta,\eta\neq0$. Let
$${\mathcal I}'=\{[\zeta\oplus\eta]\in{\mathcal I}_{(X,H)}:\zeta,\eta\neq0\}\subset{\mathcal I}_{(X,H)}.$$
It suffices to prove the following lemma.
\begin{lem}
$Q_{(X,H)}$ acts on ${\mathcal I}'$ transitively.
\end{lem}
\begin{proof}
Let $[\zeta_1\oplus\eta_1]$, $[\zeta_2\oplus\eta_2]\in{\mathcal I}'$. We are going to construct an $A\in
Q_{(X,H)}$ such that $A\cdot\zeta_1\oplus\eta_1=\zeta_2\oplus\eta_2$. Choose a decomposition $V=X\oplus H/X\oplus V/H$
and $u\in X$. Let $A_1\in GL(X)$ be the identity and take $A_2\in GL(H/X)$ which sends $\zeta_1(u)$ to $\zeta_2(u)$
and $\ker(\eta_1)$ to $\ker(\eta_2)$. Note that this is possible since $\text{Im}(\zeta_1)\subset\ker(\eta_1)$ and
$\text{Im}(\zeta_2)\subset\ker(\eta_2)$. Let $v\in H/X-\ker(\eta_2)$ and take $A_3\in GL(V/H)$ which sends
$\eta_1(A^{-1}_2(v))$ to $\eta_2(v)$. Note that for any $w\in H/X$, $A_3$ sends $\eta_1(A^{-1}_2(w))$ to $\eta_2(w)$
since $A^{-1}_2(w)\in\ker(\eta_1)$ if and only if $w\in\ker(\eta_2)$ and $\ker(\eta_1)$ is of codimension one.
Consider $A=A_1\oplus A_2\oplus A_3\in GL(V)$, it is straight forward to check that $A\cdot\zeta_1\oplus\eta_1(u\oplus
w)=\zeta_2\oplus\eta_2(u\oplus w)$ for any $u\in X$ and $w\in H/X$.
\end{proof}

To prove that $\varphi$ preserves a proper integrable distribution, we need
\begin{lem}(Lemma on p.94, \cite{mk1})\label{lg}
Let $U$ be an open set in $\proj^k$ and $\varphi:U\to\proj^k$ be an injective holomorphic map such that
$\varphi(U\cap H)$ is contained in an hyperplane for any hyperplane $H\subset\proj^k$. Then $\varphi$ extends to a
biholomorphism of $\proj^k$.
\end{lem}
The statement of the above lemma in \cite{mk1} is for $k=2$ but the proof is valid in higher dimension by replacing ''lines in $\proj^2$'' by ''hyperplanes in $\proj^k$''.
\begin{lem}\label{lf}
If $S$ is of type ($A_k,\{\alpha_1,\alpha_k\}$) and $\varphi$ preserves $\proj D^{1,0}\cup\proj D^{0,1}$, then
$\varphi$ can be extended to an automorphism of $S$.
\end{lem}
\begin{proof}
Suppose $\varphi$ preserves $\proj D^{1,0}\cup\proj D^{0,1}$. By composing with an automorphism of $S$ switching
$\proj D^{1,0}$ and $\proj D^{0,1}$, we may assume that $\varphi$ preserves $\proj D^{1,0}$ and $\proj D^{0,1}$. Since
$\varphi$ preserves $\proj D^{1,0}$ and $\proj D^{0,1}$, there exists $\varphi_1:\pi_1(U)\to\proj V$ and
$\varphi_2:\pi_2(U)\to\proj V^*$ such that $\varphi((X,H))=(\varphi_1(X),\varphi_2(H))$ for any $(X,H)\in U$, where
$U$ is the domain of $\varphi$ and $\pi_i$, $i=1,2$, are the natural projections on $S\subset\proj V\times\proj V^*$.
For any $X\in\pi_1(U)$ and $X\subset H$, $\varphi_1(X)$ is contained in the hyperplane $\varphi_2(H)$ since
$(\varphi_1(X),\varphi_2(H))\in S$. By Lemma \ref{lg}, $\varphi_1$ can be extended to a global automorphism
$\tilde\varphi_1$ of $\proj V$. Similarly, $\varphi_2$ can be extended to a global automorphism $\tilde\varphi_2$ of
$\proj V^*$. Since the global automorphism $\tilde{\varphi}:=\tilde{\varphi_1}\times\tilde{\varphi_2}$ of $\proj
V\times\proj V^*$ agrees with $\varphi$ on the open subset $U$ of $S$, $\tilde{\varphi}|_S$ is an automorphism of $S$
extending $\varphi$.
\end{proof}
\begin{prop}\label{an}
If $S$ is of type ($A_k,\{\alpha_1,\alpha_k\}$), then $\varphi$ preserves $D^{1,0}$ or $D^{0,1}$.
\end{prop}
\begin{proof}
Let $\mcc_{f(s)}$ be a variety of minimal rational tangent at $f(s)$. Suppose $df^{-1}_s(\mcc_{f(s)})$ is not
contained in $\proj D^{1,0}_s\cup\proj D^{0,1}_s$, then $df^{-1}_s(\mcc_{f(s)})$ contains ${\mathcal I}_s$ by
Proposition \ref{at}. Note that ${\mathcal I}_s$ is linearly non-degenerate and is of codimension 1 in $\proj D^1$.
Thus the set of tangent lines $\proj\wedge^2{\mathcal I}_s$ is linearly non-degenerate in $\proj\wedge^2D^1$. Then
\cite{hm1} Proposition 9 will imply that $D^1$ is integrable which is not true in this case. Therefore
$df^{-1}_s(\mcc_{f(s)})$ is contained in $\proj D^{1,0}_s\cup\proj D^{0,1}_s$ and $df^{-1}_s(\mcc_{f(s)})$ must be
equal to $\proj D^{1,0}_s$ or $\proj D^{0,1}_s$ by Lemma \ref{lf} and Proposition \ref{pb}. Hence $\varphi$ preserves
$D^{1,0}$ or $D^{0,1}$.
\end{proof}

\section*{7. End of proof}

(7.1) In the last section, we have proved that the first case of Proposition \ref{aj} can be reduced to the second case. Therefore it remains to consider the case when $\varphi$ preserves a proper equivariant integrable distribution.
\begin{lem}\label{ll}
Suppose $\varphi$ preserves two complementary equivarinat integrable distributions $D$ and $D^c$. If $f$ restricted on integral submanifolds of $D$ and $D^c$ are unramified, then $f$ is unramified.
\end{lem}
\begin{proof}
Let $\pi_1:S\to S_1$ and $\pi_2:S\to S_2$ be projections with $\ker(d\pi_1)=D$ and $\ker(d\pi_2)=D^c$. Let $l_1$, $l_2$ and $l$ be the Picard numbers of $S_1$, $S_2$ and $S$ respectively. Let $F_1$ and $F_2$ be generic fibers of $\pi_1$ and $\pi_2$. There are precisely $l_2$ highest-weight rational curves on $F_1$ and $l_1$ highest-weight rational curves on $F_2$, such that the set of $l=l_1+l_2$ rational curves are homologically independent, generating $H_2(S)$, which is of rank $l$. Now $R\cdot C=0$ for any such curve $C$ and we conclude that $R$ is cohomologically trivial, implying that $R=\emptyset$.
\end{proof}
\begin{prop}\label{pj}
If $S$ is of Picard number $l>1$ and $\varphi$ preserves a proper equivariant distribution, then there exists a
commutative diagram $(\dagger)$ as in Proposition \ref{aq} such that $f$ restricted on fibers of $\pi$ is ramified and
$\dim(X')>0$. In particular $X$ is not the projective space.
\end{prop}
\begin{proof}
We have proved that if $\varphi$ preserves a proper equivariant distribution, then $\varphi$ preserves a proper
integrable distribution $D$. We claim that one may choose $D$ so that $f$ restricted on integral submanifolds of $D$
are ramified. Suppose $f$ restricted on integral submanifolds of $D$ are unramified. Then $\varphi$ preserves an
integrable distribution $E$ such that $D\cap E=\{0\}$ by Proposition \ref{pc}. Let $H$ be the equivariant distribution generated by $D$ and $E$. If $H=T(S)$, then $E$ is the complementary distribution of $D$ and $f$ restricted on integral submanifolds of $E$ must be ramified by Lemma \ref{ll}. If $H$ is a proper distribution, replace $D$ by $H$ until $f$ restricted on integral submanifolds of $H$ are ramified. By Proposition \ref{aq}, we get the desired diagram $(\dagger)$ with $\dim(X')>0$.
\end{proof}

The following lemma says that we may assume $f'$ is bijective.
\begin{lem}\label{li}
Suppose there is a commutative diagram $(\dagger)$ as in Proposition \ref{aq}. Then there exists projective manifolds
$Y$ and $Y'$, holomorphic maps $g:S\to Y$, $g':S'\to Y'$ and $\chi:Y\to Y'$ such that the following diagram
$$\begin{CD}
S@>g>>Y\\
@V\pi VV @VV\chi V\\
S'@>g'>> Y'
\end{CD}$$
is commutative and $g'$ is bijective. Furthermore $g$ is ramified if and only if $f$ restricted on fibers of $\pi$ are
ramified.
\end{lem}
\begin{proof}
Using $S'\stackrel{f'}{\to}X'$ and $X\stackrel{\psi}{\to}X'$, take $Y=S'\times_{X'}X$ which is smooth since $X$ is smooth. Then $S\stackrel{\pi}{\rightarrow}S'$ and $S\stackrel{f}{\to}X$ give $g:S\to Y$. Take
$Y'=S'$ and let $g'$ be the identity map. The map $\chi:Y\to Y'$ is just the projection
$S'\times_{X'}X\to S'$.
\end{proof}

Now we prove that $\varphi$ cannot preserve all minimal integrable distributions.
\begin{lem}\label{lb}
Let $\pi_1:S\to S_1$ and $\pi_2:S\to S_2$ be two canonical projections on $S$ such that $\pi_1\times\pi_2:S\to
S_1\times S_2$ is injective and $S_1$ is of Picard number 1. For any $s\in S$, let $F^i_s=\pi_i^{-1}(\pi_i(s))\subset
S$, $i=1,2$. Then
$$\bigcap_{s\in\pi^{-1}_1(s_1)}d\pi_1(T_s(F^2_s))=\{0\},$$
for any $s_1\in S_1$.
\end{lem}
\begin{proof}
Define for any $s_1\in S_1$,
$$V_{s_1}=\bigcap_{s\in\pi^{-1}_1(s_1)}d\pi_1(T_s(F^2_s))\subset T_{s_1}(S_1).$$
Then $V_{s_1}$ defines an equivariant distribution on $S_1$. Write $\pi=\pi_1\times\pi_2$. If $V_{s_1}=T_{s_1}(S_1)$,
then $d\pi(T_s(S))=T_{\pi(s)}(S_1\times S_2)$ and $\pi:S\to S_1\times S_2$ is bijective. This implies that $G$ is not
simple which violates our assumption. Since $S_1$ is of Picard number one, there is no non-trivial proper equivariant
distribution on $S_1$. Therefore $V_{s_1}=\{0\}$ for all $s_1\in S_1$.
\end{proof}
\begin{lem}\label{pf}
Let $\pi_1:S\to S'_1$ and $\pi_2:S\to S'_2$ be two projections on $S$ such that $\pi_1\times\pi_2:S\to S'_1\times
S'_2$ is injective and $S'_1$ is of Picard number 1. Suppose $f:S\to X$ is a finite surjective holomorphic map and there exists
morphisms $\psi_i:X\to X_i$, $i=1,2$, onto normal varieties which induce two commutative diagrams
$$\begin{CD}
S@>f>>X\\
@V\pi_i VV @VV\psi_i V\\
S'_i@>f'_i>> X_i.
\end{CD}$$
Then $f'_1$ is unramified.
\end{lem}
\begin{proof}
Let $R_0\subset S'_1$ be the set of points such that $f'_1$ fails to be a local biholomorphism and $R_1$ be the
union of the components of $R_0$ of codimension 1. Suppose $R_1\neq\emptyset$, let $p\in R\cap\pi^{-1}_2(s_2)$ be a
point on the ramification divisor $R\subset X$ of $f$ so that $d{f'_2}_{s_2}$ is an isomorphism. For any $v\in
\ker(df_p)$, $d\pi_2(v)=0$ since $f'_2\circ\pi_2=\psi_2\circ f$. Therefore $v\in\ker(df_p)\cap\ker(d{\pi_2}_p)$.
Consider the subset $E=\{p\in S:\ker(df_p)\cap \ker(d\pi_2)\neq\{0\}\}\subset S$ of codimension one. For any $x\in E$
and non-zero $v_x\in \ker(df_x)\cap \ker(d\pi_2)$, $df'_1(d\pi_1(v_x))=d\psi_1(df(v_x))=0$ but $d\pi_1(v_x)\neq 0$ as
$v_x\in \ker(d\pi_2)$. Thus $\pi_1[E]\subset R_1$. Since $E$ and $R_1$ are of codimension one in $S$ and $S'_1$, $E$
contains some fiber $\pi_1^{-1}(s_1)$ for some $s_1\in S'_1$. We may choose $s_1$ such that
$\dim(\ker({df'_1}_{s_1}))=1$. Then for any $s\in \pi^{-1}_1(s_1)$, there exists $v_s\in\ker(d\pi_2)\subset
T_s(F^2_s)$, where $F^2_s=\pi^{-1}_2(\pi_2(s))$, such that $d\pi_1(v_s)\in \ker({df'_1}_{s_1})$. We have
$$\ker({df'_1}_{s_1})\subset\bigcap\limits_{s\in F^1_{s_1}}d{\pi_1}(T_s(F^2_s))\subset T_{s_1}(S'_1),$$
which is impossible by Lemma \ref{lb}.
\end{proof}
\begin{prop}\label{pp}
If $f$ is ramified, then $\varphi$ cannot preserve all minimal integrable distributions.
\end{prop}
\begin{proof}
Suppose $\varphi$ preserves all minimal integrable distributions. Let $D_1$ and $D_2$ be two minimal integrable
distributions. By Proposition \ref{aq}, there exists two induced maps $f'_i:S'_i\to X'_i$, $i=1,2$, and $S'_i$ are of
Picard number 1. Now $f'_i$ are injective by Lemma \ref{pf} and $f$ cannot be ramified.
\end{proof}

\textit{Proof of Main Theorem.} Let $l\geq2$ be the least positive integer such that there exists a rational homogeneous space $S$ of
Picard number $l$ and a surjective ramified holomorphic map $f:S\to X$ from $S$ to a projective manifold $X$ different from the projective space. Then the induced intertwining map $\varphi$ preserves a proper
integrable distribution $D$. So there exists a commutative diagram $(\dagger)$ such that $\ker(d\pi)=D$ and $f$
restricted on fibers of $\pi$ is ramified by Proposition \ref{pj}. We may assume that $f'$ is injective by Lemma
\ref{li}. By Proposition \ref{pc}, $\varphi$ preserves the complementary integrable distribution $D^c$ of $D$. If $l=2$, then $\varphi$ preserves the two integrable distributions $D$ and $D^c$ which is impossible by
Proposition \ref{pp}. So we may assume $l\geq3$.

If fibers of $\pi$ are of Picard number 1, then $\varphi$ preserves all minimal integrable
distributions contradicting Proposition \ref{pp}. Suppose the fibers of $\pi$ are of Picard number larger than 1, let $E\doteq Ch(D\oplus D^c)\cap D$. Since $G$ is simple, $[D,D^c]\neq 0(\bmod D\oplus D^c)$ which, together with the integrability of $D$, implies that $E\neq D$. If
$E\neq\emptyset$, then $\varphi$ preserves the proper equivariant integrable distribution $E\subset D$. If
$E=\emptyset$, then $\rank(F^{D\oplus D^c}_\eta)>0$ for any $\eta\in D$ and $\varphi$ preserves a proper equivariant distribution which is contained in the sum of
all minimal integrable distributions contained in $D$. By Proposition
\ref{pj}, there exists a finite surjective ramified holomorphic map from the fiber of $\pi$, which is of Picard number
smaller than $l$, to a projective manifold different from the projective space. This contradicts the minimality of $l$.
\end{proof}

\section*{Remark}
It is a natural question to try to extend Main Theorem to include the case where $G$ is not simple. When $G$ is not
simple, the domain $S$ is a Cartesian product, and the pull-back of variety of minimal rational tangents to $X$ can be
studied using the product structure of $S$. Using the method of Hwang-Mok \cite{hma} for the study
of holomorphic maps with smooth images on Abelian varieties, we are able to describe all surjective holomorphic maps
from $S$ to projective manifolds. This will be discussed on a forthcoming paper \cite{lau}.

\section*{Acknowledgment}
This article completed a result of the author's thesis where a further assumption that $X$ has nef tangent bundle was
made. The author would like to thank his supervisor Professor Ngaiming Mok for the guidance and support. He also
wishes to thank Professor Jun-Muk Hwang for the valuable discussions and for providing references.

\end{document}